\documentclass[10pt,reqno]{amsart}

\usepackage{enumerate}
\usepackage{amsmath, amssymb, amsthm,amsfonts}
\usepackage{mathrsfs}
\usepackage{esint}
\usepackage{xcolor}
\usepackage{mathtools}
\usepackage{hyperref}
\usepackage{bm}
\usepackage{accents}
\usepackage[foot]{amsaddr}
\usepackage{todonotes}

\makeatletter

\numberwithin{equation}{section}

\newtheorem{thm}{Theorem}[section]
\newtheorem{theorem}[thm]{Theorem}
\newtheorem{lemma}[thm]{Lemma}

\theoremstyle{definition}

\newtheorem{definition}[thm]{Definition}

\newtheorem{assumption}[thm]{Assumption}
\theoremstyle{remark}

\newtheorem{remark}[thm]{\bf{Remark}}

\newcommand\aint{-\hspace{-0.38cm}\int}

\newcommand\bH{\mathbb{H}}

\newcommand\bL{\mathbb{L}}

\newcommand\bN{\mathbb{N}}

\newcommand\bR{\mathbb{R}}

\newcommand\bZ{\mathbb{Z}}

\newcommand\cH{\mathcal{H}}

\newcommand\sL{\mathscr{L}}

\newcommand\frH{\mathfrak{H}}

\begin{document}

\title[Degenerate linear equations]{On nondivergence form linear parabolic and elliptic equations
with degenerate coefficients}

\author{Hongjie Dong$^{1}$}
\address{$^1$ Division of Applied Mathematics, Brown University, 182 George Street, Providence, RI 02912, USA}
\email{Hongjie\_Dong@brown.edu}
\thanks{H. Dong was partially supported by the NSF under agreement DMS2350129.}

\author{Junhee Ryu$^{2,*}$}
\address{$^2$ School of Mathematics, Korea Institute for Advanced Study, 85 Hoegi-ro, Dongdaemun-gu, Seoul, 02455, Republic of Korea}
\email{junhryu@kias.re.kr}
\thanks{$^*$Corresponding author. J. Ryu was supported by a KIAS Individual Grant (MG101501) at Korea Institute for Advanced Study and
by the National Research Foundation of Korea (NRF) grant funded by the Korea government (MSIT) (No. RS-2026-25477288).}



\begin{abstract}
We establish the unique solvability in weighted mixed-norm Sobolev spaces for a class of degenerate parabolic and elliptic equations in the upper half space. The operators are in nondivergence form, with the leading coefficients given by  $x_d^2a_{ij}$, where $a_{ij}$ is bounded, uniformly nondegenerate, and measurable in $(t,x_d)$ except $a_{dd}$, which is measurable in $t$ or $x_d$. In the remaining spatial variables, they have weighted small mean oscillations. In addition, we investigate the optimality of the function spaces associated with our results.
\end{abstract}

\maketitle

\section{Introduction}

In this paper, we study a class of parabolic and elliptic equations with degenerate coefficients in
the upper half space. More precisely, we investigate
\begin{equation} \label{mainpara}
  \sL_p u +\lambda c_0 u = f
\end{equation}
and
\begin{equation} \label{mainell}
  \sL_e u +\lambda c_0 u = f,
\end{equation}
where $\sL_p$ and $\sL_e$ are the second-order nondivergence form operators defined by
\begin{equation*}
  \sL_p u:= a_0 u_t - x_d^2 a_{ij}D_{ij}u + x_d b_iD_iu + cu
\end{equation*}
and
\begin{equation*}
  \sL_e u:= -x_d^2 a_{ij}D_{ij}u + x_d b_iD_iu +cu,
\end{equation*}
respectively.
Here, $\lambda\geq0$, $f$ is a given measurable forcing term, $(a_{ij})$ satisfies the ellipticity and boundedness conditions
\begin{equation} \label{ellip}
  \nu|\xi|^2\leq a_{ij}\xi_i\xi_j, \quad |a_{ij}|\leq \nu^{-1},
\end{equation}
and $a_0$, $(b_i)$, $c$, and $c_0$ satisfy
\begin{equation} \label{lowerbound}
  |b_i|,|c| \leq K
\end{equation}
and
\begin{equation} \label{rembound}
  K^{-1}\leq a_0,c_0 \leq K.
\end{equation}
We emphasize that the two zeroth-order terms, $cu$ and $\lambda c_0 u$, play distinct roles in the analysis.
The domain of the parabolic problem \eqref{mainpara} is $\Omega_T := (-\infty,T)\times\bR^d_+$, where $T\in(-\infty,\infty]$, $d\in\bN$, and $\bR^{d}_+ := \{(x_1,\dots,x_d)\in\bR^d : x_d>0\}$. We also consider the corresponding Cauchy problems in $\Omega_{0,T}:=(0,T)\times \bR^d_+$. For the elliptic problem \eqref{mainell}, the domain is the upper half space $\bR^d_+$.

Observe that equations \eqref{mainpara} and \eqref{mainell} are stated without boundary conditions, as their solvability in weighted Sobolev spaces does not require such conditions.
Nevertheless, it is possible to study these equations under certain boundary conditions. In such cases, additional assumptions on the forcing term $f$ or on the structure of the equations are necessary (see e.g. \cite[Remark 4.7]{V89}). In this paper, however, our goal is to address the equations under minimal assumptions.

The equations are motivated by several models. One of the most elementary examples is the Euler equation, which is a classical ordinary differential equation: 
\begin{equation} \label{eq8192324}
    -ax^2D_{xx}u + xbD_xu + cu =f \, \text{ in } \, \bR_+,
\end{equation}
where $a,b,c\in\bR$ are constants. For the parabolic equation \eqref{mainpara} with $d=1$, one can refer to the Black--Scholes--Merton equation
\begin{equation*}
    u_t + \frac{1}{2}\sigma^2x^2D_{xx}u + rxD_xu -ru=0 \, \text{ in } \, \Omega_{0,T},
\end{equation*}
where $\sigma,r>0$ are constants.
We are also motivated by degenerate viscous Hamilton--Jacobi equations of the form
\begin{equation*}
    u_t(t,x)-x_d^{\alpha}\Delta u(t,x) +\lambda u(t,x) + H(t,x,D_xu)=0 \text{ in } \Omega_T,
\end{equation*}
where $\alpha>0$, and $H:\Omega_T\times \bR^d\to\bR$ is a given smooth Hamiltonian. Here, equation \eqref{mainpara} corresponds to the special case $H=0$ and $\alpha=2$. Hamilton–Jacobi equations play an important role in optimal control, differential games, and related areas. In particular, regularity theory is a fundamental subject in nonlinear PDEs. We refer the reader to \cite{AT15, DPT24} for more information. Furthermore, the equations are derived from a variety of problems. For instance, \eqref{mainell} is obtained in the linearization of the nonlinear and degenerate Loewner--Nirenberg problem
\begin{equation*}
    \Delta u=\frac{d(d-2)}{4}u^{\frac{d+2}{d-2}} \, \text{ in } \, \bR^d_+.
\end{equation*}
For further derivations in other fields, see \cite{HX24,V89} and the references therein.

Let us illustrate the effect of the coefficients by considering the simplest equation \eqref{eq8192324}. Suppose that the quadratic polynomial
\begin{equation} \label{eq3222323}
    az^2+(b+a)z-c=0,
\end{equation}
which involves both leading and lower-order coefficients, has two distinct real roots $\alpha$ and $\beta$. Then, it is well-known that the general solution of \eqref{eq8192324} is
\begin{equation*}
  u(x)=(A_1(x)+B_1)x^{-\alpha} + (A_2(x)+B_2)x^{-\beta},
\end{equation*}
where $B_1$ and $B_2$ are arbitrary constants, and
\begin{align*}
  A_1(x) := -\frac{1}{a(\beta-\alpha)} \int_0^x y^{\alpha-1} f(y) \, dy, \quad A_2(x) := \frac{1}{a(\beta-\alpha)} \int_0^x y^{\beta-1} f(y) \, dy,
\end{align*}
which shows that the boundary behavior of the solution is influenced not only by $a$ but also by $b$ and $c$.
In other words, unlike nondegenerate equations, the lower-order coefficients must be handled carefully. Moreover, a solution may fail to exist in the standard $L_p(\bR^d_+)$ space. 

In order to address the aforementioned issues, we introduce a class of weights. We present maximal regularity results for solutions in weighted Sobolev spaces. In particular, for the elliptic case \eqref{mainell}, we prove
\begin{equation} \label{eq8202312}
      \int_{\bR^d_+} \left(|(1+\lambda)u|^p + |(1+\sqrt{\lambda})x_dD_xu|^p + |x_d^2D_x^2u|^p \right) x_d^{\theta-1} dx \leq N \int_{\bR^d_+} |f|^p x_d^{\theta-1} dx,
\end{equation}
by considering two distinct cases for $\lambda\geq0$ and $\theta\in\bR$. 
\\
\textbf{Case 1} The case $\lambda=0$ can be treated under the restricted range of $\theta$. In this case, the lower-order coefficients $(b_i)$ and $c$ are ``effective'' in the sense that the admissible range of $\theta$ is determined by the following ratios of coefficients:
\begin{equation*}
  \frac{b_d}{a_{dd}}=n_b, \quad \frac{c}{a_{dd}}=n_c.
\end{equation*}
The precise range is given in Theorem \ref{thmell}, which is optimal in the sense that it is a necessary and sufficient condition for the solvability. The proof of this optimality is presented in Section \ref{sectheta}. 
\\
\textbf{Case 2}
When $\lambda\geq0$ is chosen sufficiently large, one may consider any $\theta\in\bR$ without restriction.

For the parabolic problem \eqref{mainpara}, an analogous result is established in weighted mixed-norm spaces. In the spatial variables, we consider the same class of weights as above, whereas in the time variable, we allow general Muckenhoupt weights.

Regarding $a_{ij}$ in the leading coefficients, we consider a large class of functions. They are assumed to be measurable in $(t,x_d)$ except $a_{dd}$, which is assumed to be measurable in $t$ or $x_d$. In the remaining spatial directions, they have small bounded mean oscillations. For the classical (nondegenerate) heat equation, this class is known to be optimal in the sense that if $a_{ij}$ are merely measurable in $(t,x_d)$ for all $i,j=1,\dots,d$, then there is no unique solvability of parabolic equations. We refer the reader to \cite{K16} for a counterexample, and to \cite{D20,DK18} (and the references therein) for results on the solvability of nondegenerate equations. In this paper, we adopt the same class of coefficients.

We now review related literature on \eqref{mainpara} and \eqref{mainell}.
First, when it comes to $L_p$ regularity results, we refer the reader to \cite{Khalf,K07,MNS22,MNS24}, where the leading coefficients are uniformly continuous. See also \cite{FMP12,V89}, where the estimate \eqref{eq8202312} with $\theta=1$ was obtained. In particular, in these papers, operators of the form $x_d^\alpha\Delta$ with $\alpha\geq2$ were considered.
Recently, in \cite{DR24div}, we studied the divergence form equations corresponding to \eqref{mainpara} and \eqref{mainell}. In both \cite{DR24div} and the present paper, we deal with a substantially larger class of coefficients in weighted mixed-norm spaces than in the earlier works. In particular, in the present paper, we allow a more general assumption on the leading coefficients than \cite{DR24div} by exploiting the structure of the nondivergence form equations (see Remark \ref{rem8202332}). In addition, we establish the optimality of the results, which can also be applied to \cite{DR24div}.

Next, we describe regularity results in H\"older spaces. 
In \cite{V89, GL91, HX24, HX25}, the following equations were studied in a bounded domain instead of in the half space:
\begin{equation*}
\rho^{2s}a_{ij}D_{ij}u+\rho^{s}b_iD_iu+cu=f,
\end{equation*}
where $s\geq1$ and $\rho$ is a regularized distance function. Note that the case $s=1$ is a version of our equations. In \cite{V89}, weighted H\"older theory was introduced for all $s\geq1$. In \cite{GL91,HX24,HX25}, higher-order weighted regularity of solutions was obtained when $s=1$. 

Below we give a brief review on equations with lower-order degeneracy. In \cite{DPT23, DPT24}, for $\alpha<2$, equations with the prototype
\begin{equation*}
  u_t-x_d^\alpha\Delta u +\lambda u=f, \quad u(x',0)=0,
\end{equation*}
 were studied in certain weighted Sobolev spaces. 
 We remark that although \eqref{mainpara} is the limiting case as $\alpha\to2$, our main results cannot be derived by (formally) taking the limit in their results.
 In \cite{LY25, Y25}, H\"older regularity of solutions was obtained when $\alpha<1$.
 See also \cite{DH98, FP13,FMPP07,Koch} for the special case $\alpha=1$.

Lastly, we refer the reader to a series of papers \cite{DP20,DP21TAMS, DP23,DP23JFA,BD24}, where the weights of coefficients of $u_t$ and $D_x^2u$ appear in a balanced way, which plays a crucial role in the analysis and functional space settings.
In particular, in \cite{DP20,DP23JFA}, the elliptic equations of the prototype
\begin{equation*}
  x_d^2\Delta u+\alpha x_d D_du-\lambda x_d^2u=f
\end{equation*}
with boundary conditions were also studied. Note that here the zeroth-order term is $\lambda x_d^2u$ rather than $\lambda u$, which makes boundary conditions necessary.

The paper is organized as follows. In Section \ref{sec_main}, we introduce the weighted Sobolev spaces, assumptions, and our main results.
In Section \ref{secpara}, we provide the proofs of the main results on the parabolic problem (Theorems \ref{thmpara} and \ref{thmfinite}). In Section \ref{sec_ell}, we prove Theorem \ref{thmell}, which concerns the elliptic problem. We address the optimality of our main results in Section \ref{sectheta}.
Lastly, in Section \ref{sec_app}, we show that our main results can be applied to degenerate viscous Hamilton–Jacobi equations when the Hamiltonian is of certain specific forms.

\section{Main results} \label{sec_main}

\subsection{Notation and function spaces} \label{sec_function}

We begin this section by introducing some notation used throughout the paper.
 We use ``$:=$'' or ``$=:$'' to denote a definition.
For nonnegative functions $f$ and $g$, we write $f\approx g$ if there exists a constant $N>0$ such that $N^{-1}f\leq g\leq Nf$.
  By $\bN$, we denote the natural number system. We write $\bN_0:=\bN\cup\{0\}$. As usual, $\bR^d$ stands for the Euclidean space of points $x=(x_1,\dots,x_d)=(x',x_d)$, and we write $\bR:=\bR^1$.
We use $D^n_x u$ to denote the partial derivatives of order $n\in\bN_0$ with respect to the space variables, and $D_xu:=D_x^1u$. We also denote
\begin{equation*}
  D_iu=\frac{\partial u}{\partial x_i}, \quad D_{ij}u=\frac{\partial^2 u}{\partial x_i \partial x_j}.
\end{equation*}

Next, we introduce function spaces. For $p\in[1,\infty)$ and $\theta\in\bR$, we denote by $L_{p,\theta}=L_{p,\theta}(\bR^d_+)$ the set of all measurable functions $u$ defined on $\bR^d_+$ such that
\begin{equation*}
  \|u\|_{L_{p,\theta}}:=\left( \int_{\bR^d_+} |u|^p x_d^{\theta-1} dx \right)^{1/p}<\infty.
\end{equation*}
For $n\in \bN$, we denote the weighted Sobolev space
\begin{equation*}
  H_{p,\theta}^n:=\{u: u, x_dD_xu, \dots, x_d^n D_x^nu \in L_{p,\theta}\},
\end{equation*}
where the norm in $H_{p,\theta}^n$ is given by
\begin{equation} \label{eq2102313}
  \|u\|_{H_{p,\theta}^n} := \left( \sum_{i=0}^n \int_{\bR^d_+}   |x_d^iD_x^iu|^p x_d^{\theta-1} dx \right)^{1/p}.
\end{equation}
We remark that $H_{p,\theta}^n$ corresponds to $H_{p,\theta+d-1}^n$ in \cite{Khalf}.  Moreover, from \cite[Corollary 2.3]{Khalf}, the norm \eqref{eq2102313} is equivalent to
\begin{equation} \label{equivnorm}
  \|u\|_{H_{p,\theta}^n} \approx \left(\sum_{m=-\infty}^\infty e^{m(\theta+d-1)} \|u(e^m\cdot)\zeta\|_{W_p^n(\bR^d)}^p\right)^{1/p},
\end{equation}
where $W_p^n(\bR^d)$ is the (unweighted) Sobolev space of order $n$ and $\zeta\in C_c^\infty(\bR_+)$ is a nonnegative function satisfying
\begin{equation} \label{eq2102340}
  \sum_{n=-\infty}^\infty \zeta^p(e^{x_d-n})\geq1.
\end{equation}

To the best of our knowledge, the spaces $H_{p,\theta}^n$ were firstly introduced in \cite[Section 2.6.3]{ML68} for $p=2$ and $\theta=1$. They were generalized in a unified manner for $p\in(1,\infty)$ and $\theta\in\bR$ in \cite{Khalf} to study stochastic partial differential equations (SPDEs). We refer the reader to \cite{K94,KLline,KLspace}. 
See also \cite{KK2004,KN09,DK15,S24,LLRV25}, where second-order nondegenerate equations are studied.

We now turn to the definitions of mixed-norm spaces.
Let $p,q\geq1$, and $\omega=\omega(t)$ be a weight on $(-\infty,T)$. For functions defined on $\Omega_T=(-\infty,T)\times \bR^d_+$, define
 \begin{equation*}
   \bL_{q,p,\theta,\omega}(T):=L_q((-\infty,T),\omega dt;L_{p,\theta}), \quad \bH_{q,p,\theta,\omega}^n(T):=L_q((-\infty,T), \omega dt;H_{p,\theta}^n).
 \end{equation*}
 Similarly, for functions on $\Omega_{0,T}=(0,T)\times\bR^d_+$, we define
  \begin{equation*}
   \bL_{q,p,\theta,\omega}(0,T):=L_q((0,T),\omega dt;L_{p,\theta}), \quad \bH_{q,p,\theta,\omega}^n(0,T):=L_q((0,T), \omega dt;H_{p,\theta}^n).
 \end{equation*}

 For parabolic equations, we denote $u\in \frH_{q,p,\theta,\omega}^2(T)$ if
 \begin{equation*}
     u,x_dD_xu,x_d^2D_x^2u, u_t \in \bL_{q,p,\theta,\omega}(T),
 \end{equation*}
 and set
 \begin{equation*}
    \|u\|_{\frH_{q,p,\theta,\omega}^2(T)} = \|u\|_{\bH_{q,p,\theta,\omega}^2(T)} + \|u_t\|_{\bL_{q,p,\theta,\omega}(T)}.
 \end{equation*}
By \cite[Theorem 1.19, Remark 5.5]{Khalf}, the spaces $C_c^\infty(\bR^d_+)$, $C_c^\infty((-\infty,T)\times\bR^d_+)$, and $C_c^\infty((-\infty,T]\times\bR^d_+)$ are dense in $H_{p,\theta}^n$, $\bH_{q,p,\theta,\omega}^n(T)$, and $\frH_{q,p,\theta,\omega}^2(T)$, respectively.
For the Cauchy problems, the space $\mathring{\frH}_{q,p,\theta,\omega}^2(0,T)$ is defined by the closure of the set of functions $u\in C_c^\infty([0,T]\times \bR^d_+)$ with $u(0,\cdot)=0$, equipped with the norm
\begin{equation*}
\|u\|_{\mathring{\frH}_{q,p,\theta,\omega}^2(0,T)}:=\|u\|_{\bH_{q,p,\theta,\omega}^2(0,T)} +  \|u_t\|_{\bL_{q,p,\theta,\omega}(0,T)}.
\end{equation*}

In our main result, we consider weights from the $A_q$ Muckenhoupt class.

\begin{definition}
  For $q\in(1,\infty)$, $\omega:\bR \to [0,\infty)$ is said to be in the $A_q(\bR)$ Muckenhoupt class of weights if
  \begin{equation*}
    [\omega]_{A_q(\bR)}:=[\omega]_{A_q}:=\sup_{r>0, t\in\bR} \left( \aint_{t-r}^{t} \omega(s)ds \right)\left( \aint_{t-r}^{t} \omega(s)^{-1/(q-1)} ds \right)^{q-1} <\infty.
  \end{equation*}
\end{definition}

\subsection{Parabolic equations}

In this subsection, we consider parabolic equations.

First, we state the regularity assumptions on the coefficients. The parameters $\rho_0\in(1/2,1)$ and $\gamma_0>0$ will be determined later.

\begin{assumption}[$\rho_0,\gamma_0$] \label{ass_lead}
For every $x_0\in \bR^d_+$ and $\rho\in(0,\rho_0x_{0d}]$, there exist coefficients $[a_{0}]_{\rho, x_0}, [a_{ij}]_{\rho,x_0}$, and $[c_{0}]_{\rho,x_0}$ satisfying \eqref{ellip} and \eqref{rembound}. Moreover,
\begin{itemize}
  \item $[a_{0}]_{\rho, x_0}, [a_{dd}]_{\rho, x_0}$, and $[c_0]_{\rho, x_0}$ -- uniformly in $\rho$ -- either:
  \begin{itemize}
    \item depend only on $x_d$, or 
    \item depend only on $t$,
  \end{itemize}

  \item $[a_{ij}]_{\rho,x_0}$ depend only on $(t,x_d)$ for $(i,j)\neq(d,d)$,

  \item for any $t\in(-\infty,T)$,
   \begin{align} \label{eq2111512}
    &\aint_{B_\rho(x_0)} \Big( \left| a_0(t,y)- [a_0]_{\rho,x_0}(t,y_d) \right| + \left| a_{ij}(t,y)- [a_{ij}]_{\rho,x_0}(t,y_d) \right| \nonumber
    \\
    &\qquad\qquad + \left| c_0(t,y)- [c_0]_{\rho,x_0}(t,y_d) \right| \Big) \, dy < \gamma_0.
  \end{align}
\end{itemize}
\end{assumption}

The following stronger assumption will be imposed to consider the case when $\lambda=0$.

\begin{assumption}[$\rho_0,\gamma_0$] \label{ass_coeff}
For every $x_0\in \bR^d_+$ and $\rho\in(0,\rho_0x_{0d}]$, there exist coefficients $[a_{0}]_{\rho,x_0}, [a_{ij}]_{\rho,x_0}, [b_{i}]_{\rho,x_0}, [c]_{\rho,x_0}$, and $[c_0]_{\rho,x_0}$ satisfying \eqref{ellip}-\eqref{rembound} and the ratio condition
    \begin{eqnarray*}
    \frac{[b_{d}]_{\rho,x_0}}{[a_{dd}]_{\rho,x_0}}=n_b, \quad \frac{[c]_{\rho,x_0}}{[a_{dd}]_{\rho,x_0}}=n_{c}
  \end{eqnarray*}
  for some $n_b, n_c\in\bR$ independent of $x_0$ and $\rho$.
  Moreover,
\begin{itemize}
  \item $[a_{0}]_{\rho, x_0}, [a_{dd}]_{\rho, x_0}, [b_d]_{\rho, x_0}, [c]_{\rho, x_0}$, and $[c_0]_{\rho, x_0}$  -- uniformly in $\rho$ -- either:
  \begin{itemize}
    \item depend only on $x_d$, or
    
    \item depend only on $t$,
  \end{itemize}

  \item $[a_{ij}]_{\rho,x_0}$ depend only on $(t,x_d)$ for $(i,j)\neq (d,d)$,

  \item $[b_{i}]_{\rho,x_0}$ depend only on $(t,x_d)$ for $i\neq d$,

  \item for any $t\in(-\infty,T)$,
  \begin{align} \label{eq2111518}
  &\aint_{B_\rho(x_0)} \Big( \left| a_{0}(t,y)- [a_{0}]_{\rho,x_0}(t,y_d) \right| + \left| a_{ij}(t,y)- [a_{ij}]_{\rho,x_0}(t,y_d) \right|  \nonumber
    \\
    &\qquad\qquad+ \left| b_{i}(t,y)- [b_{i}]_{\rho,x_0}(t,y_d) \right| + \left| c(t,y)- [c]_{\rho,x_0}(t,y_d) \right|  \nonumber
        \\
    &\qquad\qquad + \left| c_0(t,y)- [c_0]_{\rho,x_0}(t,y_d) \right| \Big) \, dy < \gamma_0.
  \end{align}
\end{itemize}
\end{assumption}

In \eqref{eq2111512} and \eqref{eq2111518}, the variables of the coefficients should be interpreted appropriately. For instance, $[a_{dd}]_{\rho,x_0}(t,y_d)$ can be either $[a_{dd}]_{\rho,x_0}(y_d)$ or $[a_{dd}]_{\rho,x_0}(t)$.

Now we state the first main result of our paper, which addresses \eqref{mainpara} on $\Omega_T$.

\begin{theorem} \label{thmpara}
  Let $T\in(-\infty,\infty]$, $p,q\in(1,\infty)$, $\theta\in\bR$, $\omega\in A_q(\bR)$, and $[\omega]_{A_q}\leq K_0$. Suppose that \eqref{ellip}-\eqref{rembound} are satisfied.

$(i)$ There exist
 \begin{eqnarray*}
   \rho_0=\rho_0(d,p,q,\theta,\nu,K,K_0)\in(1/2,1)
 \end{eqnarray*}
 sufficiently close to $1$, a sufficiently small number
 \begin{eqnarray*}
   \gamma_0=\gamma_0(d,p,q,\theta,\nu,K,K_0)>0,
 \end{eqnarray*}
 and a sufficiently large number
  \begin{eqnarray*}
   \lambda_0=\lambda_0(d,p,q,\theta,\nu,K,K_0)\geq 0,
 \end{eqnarray*}
 such that if Assumption \ref{ass_lead} $(\rho_0,\gamma_0)$ is satisfied, then the following assertions hold true.
  For any $\lambda\geq \lambda_0$ and $f \in \bL_{q,p,\theta,\omega}(T)$, there is a unique solution $u\in \frH_{q,p,\theta,\omega}^2(T)$ to \eqref{mainpara}.
Moreover, for this solution, we have
\begin{eqnarray} \label{eq2111600}
  &\|u_t\|_{\bL_{q,p,\theta,\omega}(T)} + (1+\lambda)\|u\|_{\bL_{q,p,\theta,\omega}(T)}& \nonumber
  \\
  &+ (1+\sqrt{\lambda}) \|x_dD_xu\|_{\bL_{q,p,\theta,\omega}(T)} + \|x_d^2D_x^2u\|_{\bL_{q,p,\theta,\omega}(T)}& \nonumber
  \\
  &\leq N \|f\|_{\bL_{q,p,\theta,\omega}(T)},&
\end{eqnarray}
where $N=N(d,p,q,\theta,\nu,K,K_0)$.

$(ii)$ Let $n_b,n_c\in\bR$. Suppose that the quadratic equation 
\begin{equation} \label{eq2152256}
  z^2+(1+n_b)z-n_c=0
\end{equation}
has two distinct real roots $\alpha<\beta$. Then there exist
 \begin{eqnarray*}
   \rho_0=\rho_0(d,p,q,\theta,n_b,n_c,\nu,K,K_0)\in(1/2,1)
 \end{eqnarray*}
 sufficiently close to $1$, and a sufficiently small number
 \begin{eqnarray*}
   \gamma_0=\gamma_0(d,p,q,\theta,n_b,n_c,\nu,K,K_0)>0
 \end{eqnarray*}
 such that under Assumption \ref{ass_coeff} $(\rho_0,\gamma_0)$, the assertions in $(i)$ hold with $\lambda_0=0$ and $\alpha p<\theta<\beta p$, where the dependencies of the constant $N$ are replaced by $d,p,q,\theta,n_b,n_c,\nu,K$, and $K_0$.
\end{theorem}

\begin{remark}
    We provide some comments on the components in the statement of Theorem \ref{thmpara}.

    $(i)$ The parameter $\rho_0$ is closely related to the construction of a partition of unity in the weighted spaces. Unlike the unweighted case, weighted derivatives of the partition of unity can be made arbitrarily small by choosing $\rho_0$ sufficiently close to $1$ (see \eqref{eq2141547}). This smallness is essential for controlling the localization errors in the a priori estimates.

$(ii)$ The smallness of $\gamma_0$ implies that the coefficients have small bounded mean oscillations in the relevant spatial variables.

$(iii)$ The range of $\theta$ in $(ii)$ of Theorem \ref{thmpara} is determined by $n_b$ and $n_c$. Roughly, as in the ODE case \eqref{eq8192324}, the behavior of the solution is influenced by $a_{dd}, b_d$, and $c$. We also note that \eqref{eq2152256} is equivalent to \eqref{eq3222323} when $d=1$ and the coefficients are constant.

    $(iv)$ In $(i)$ of the theorem, the lower order coefficients $(b_i)$ and $c$ are assumed to be merely bounded measurable, and $\theta\in\bR$ is arbitrary, although we need to choose a sufficiently large $\lambda_0\geq0$.
On the other hand, in $(ii)$, we can take $\lambda=0$ by imposing ratio conditions on the lower-order coefficients and restricting the range of $\theta$.
\end{remark}

\begin{remark} \label{rem8202332}
In \cite[Theorem 2.6 $(ii)$]{DR24div}, the corresponding divergence form equation was studied under the assumption that either

        $-$ $a_0$ depends only on $x_d$ and $[a_{dd}]_{\rho,x_0}$ is constant, or

    $-$ both $a_0$ and $[a_{dd}]_{\rho,x_0}$ depend only on $t$.
 \\   
In contrast, in Theorem \ref{thmpara}, we consider a more general case. For instance, we allow that $[a_{dd}]_{\rho,x_0}$ depends only on $x_d$ and $a_0$ has small bounded mean oscillations in the other spatial variables.
This generalization is possible because we can divide the equation by $a_{dd}$ and thus reduce it to the constant case $a_{dd}=1$. In addition, in our setting,  $a_{0}u_{t}$ belongs to $L_{p}$ space rather than to a negative Sobolev space, which enables us to treat more general class of $a_0$.
\end{remark}

Next we present the Cauchy problems
\begin{equation} \label{mainfinite}
  a_0 u_t - x_d^2 a_{ij}D_{ij}u + x_d b_iD_iu + cu +\lambda c_0 u = f
\end{equation}
in $\Omega_{0,T}:=(0,T)\times\bR^d_+$ with $u(0,\cdot)=0$. In the following result, we handle arbitrary $\lambda\geq0$ by allowing the constant $N$ to depend on $T$.

\begin{theorem} \label{thmfinite}
  Let $T\in(0,\infty)$, $p,q\in(1,\infty)$, $\theta\in\bR$, $\omega\in A_q(\bR)$, and $[\omega]_{A_q}\leq K_0$. Suppose that \eqref{ellip}-\eqref{rembound} are satisfied.
 Then there exist
 \begin{eqnarray*}
\rho_0=\rho_0(d,p,q,\theta,\nu,K,K_0)\in(1/2,1)
 \end{eqnarray*}
 sufficiently close to $1$, and a sufficiently small number
 \begin{eqnarray*}
\gamma_0=\gamma_0(d,p,q,\theta,\nu,K,K_0)>0,
 \end{eqnarray*}
 such that if Assumption \ref{ass_lead} $(\rho_0,\gamma_0)$ is satisfied, then the following assertions hold.
 For any $\lambda\geq 0$ and $f \in \bL_{q,p,\theta,\omega}(T)$, there is a unique solution $u\in \mathring{\frH}_{q,p,\theta,\omega}^2(0,T)$ to \eqref{mainfinite} with zero initial condition $u(0,\cdot)=0$.
Moreover, for this solution, we have
\begin{eqnarray} \label{eq2121343}
  &\|u_t\|_{\bL_{q,p,\theta,\omega}(0,T)} + (1+\lambda)\|u\|_{\bL_{q,p,\theta,\omega}(0,T)}& \nonumber
  \\
  &+ (1+\sqrt{\lambda}) \|x_dD_xu\|_{\bL_{q,p,\theta,\omega}(0,T)} + \|x_d^2D_x^2u\|_{\bL_{q,p,\theta,\omega}(0,T)}& \nonumber
  \\
  &\leq N \|f\|_{\bL_{q,p,\theta,\omega}(0,T)},&
\end{eqnarray}
where $N=N(d,p,q,\theta,\nu,K,K_0,T)$.
\end{theorem}

\subsection{Elliptic equations}

In this subsection, we present our regularity assumptions on the coefficients and state the main result for elliptic equations.

As in the parabolic case, the parameters $\rho_0\in(1/2,1)$ and $\gamma_0>0$ will be specified later.

\begin{assumption}[$\rho_0,\gamma_0$] \label{ass_ellip_lead}
For every $x_0 \in \bR^d_+$ and $\rho\in(0,\rho_0x_{0d}]$, there exist coefficients $[a_{ij}]_{\rho,x_0}$ and $[c_0]_{\rho,x_0}$ satisfying \eqref{ellip} and \eqref{rembound}.
  Moreover,
\begin{itemize}
    \item $[a_{ij}]_{\rho, x_0}$ and $[c_0]_{\rho, x_0}$ depend only on $x_d$,

  \item we have
  \begin{align*}
    &\aint_{B_\rho(x_0)} \Big( \left| a_{ij}(y)- [a_{ij}]_{\rho,x_0}(y_d) \right| + \left| c_0(y)- [c_0]_{\rho,x_0}(y_d) \right| \Big) \, dy < \gamma_0.
  \end{align*}
\end{itemize}
\end{assumption}

We impose the following stronger assumption on the coefficients for the case when $\lambda=0$.

\begin{assumption}[$\rho_0,\gamma_0$] \label{ass_ellip}
For every $x_0\in \bR^d_+$ and $\rho\in(0,\rho_0x_{0d}]$, there exist coefficients $[a_{ij}]_{\rho,x_0}, [b_{i}]_{\rho,x_0}, [c]_{\rho,x_0}$, and $[c_0]_{\rho,x_0}$ satisfying \eqref{ellip}-\eqref{rembound} and the ratio condition
    \begin{eqnarray*}
    \frac{[b_{d}]_{\rho,x_0}}{[a_{dd}]_{\rho,x_0}}=n_b, \quad \frac{[c]_{\rho,x_0}}{[a_{dd}]_{\rho,x_0}}=n_{c}
  \end{eqnarray*}
  for some $n_b, n_c\in\bR$ independent of $x_0$ and $\rho$.
  Moreover,
\begin{itemize}
    \item $[a_{ij}]_{\rho, x_0}, [b_i]_{\rho, x_0}, [c]_{\rho, x_0}$, and $[c_0]_{\rho,x_0}$ depend only on $x_d$,
    
  \item we have
  \begin{align*}
    &\aint_{B_\rho(x_0)} \Big( \left| a_{ij}(y)- [a_{ij}]_{\rho,x_0}(y_d) \right| + \left| b_{i}(y)- [b_{i}]_{\rho,x_0}(y_d) \right| \nonumber
    \\
    &\qquad\qquad+ \left| c(y)- [c]_{\rho,x_0}(y_d) \right| + \left| c_0(y)- [c_0]_{\rho,x_0}(y_d) \right| \Big) \, dy < \gamma_0.
  \end{align*}
\end{itemize}
\end{assumption}

\begin{theorem} \label{thmell}
Let $p\in(1,\infty)$ and $\theta\in\bR$.  Suppose that \eqref{ellip}-\eqref{rembound} are satisfied.

$(i)$ There exist
 \begin{eqnarray*}
   \rho_0=\rho_0(d,p,\theta,\nu,K)\in(1/2,1)
 \end{eqnarray*}
 sufficiently close to $1$, a sufficiently small number
 \begin{eqnarray*}
   \gamma_0=\gamma_0(d,p,\theta,\nu,K)>0,
 \end{eqnarray*}
 and a sufficiently large number
  \begin{eqnarray*}
   \lambda_0=\lambda_0(d,p,\theta,\nu,K)\geq 0,
 \end{eqnarray*}
 such that under Assumption \ref{ass_ellip_lead} $(\rho_0,\gamma_0)$, the following assertions hold.
For any $\lambda\geq \lambda_0$ and $f \in L_{p,\theta}$, there is a unique solution $u\in H_{p,\theta}^2$ to \eqref{mainell}.
Moreover, for this solution, we have
\begin{align} \label{eq216207}
  & (1+\lambda)\|u\|_{L_{p,\theta}} + (1+\sqrt{\lambda})\|x_dD_xu\|_{L_{p,\theta}} + \|x_d^2D_x^2u\|_{L_{p,\theta}} \leq N \|f\|_{L_{p,\theta}},
\end{align}
where $N=N(d,p,\theta,\nu,K)$.

$(ii)$ Let $n_b,n_c\in\bR$. Suppose that the quadratic equation \eqref{eq2152256} has two distinct real roots $\alpha<\beta$. Then there exist
 \begin{eqnarray*}
\rho_0=\rho_0(d,p,\theta,n_b,n_c,\nu,K)\in(1/2,1)
 \end{eqnarray*}
 sufficiently close to $1$, and a sufficiently small number
 \begin{eqnarray*}
\gamma_0=\gamma_0(d,p,\theta,n_b,n_c,\nu,K)>0
 \end{eqnarray*}
such that under Assumption \ref{ass_ellip} $(\rho_0,\gamma_0)$, the assertions in $(i)$ hold with $\lambda_0=0$ and $\alpha p<\theta<\beta p$, where the dependencies of the constant $N$ are replaced by $d,p,\theta,n_b,n_c,\nu$, and $K$.

$(iii)$ When $d=1$ and $\lambda=0$, the assertions in $(ii)$ hold true for $\theta\in \bR\setminus\{\alpha p,\beta p\}$.
\end{theorem}

\begin{remark}
By dividing \eqref{mainell} by $x_d^2$, the equation can be viewed as a nondegenerate equation with singular lower-order coefficients satisfying a growth condition. Such equations were studied in \cite[Theorem 2.3]{DK15}, which corresponds to a special case of Theorem \ref{thmell} with  $n_b=n_c=0$,
\end{remark}

\section{Parabolic equations} \label{secpara}

In this section, we deal with parabolic equations and give the proofs of Theorems \ref{thmpara} and \ref{thmfinite}.

To prove our main results, the first step is to obtain higher-order derivative estimates for solutions when $p=q$. We write $\bL_{p,\theta,\omega}(T):=\bL_{p,p,\theta,\omega}(T)$, $\bH^n_{p,\theta,\omega}(T):=\bH^n_{p,p,\theta,\omega}(T)$, and $\frH^2_{p,\theta,\omega}(T):=\frH^2_{p,p,\theta,\omega}(T)$.

\begin{lemma} \label{lem2171542}
Let $\rho_0\in(1/2,1)$, $\lambda\geq0$, $T\in(-\infty,\infty]$, $p\in(1,\infty)$, $\theta\in \bR$, $\omega\in A_p(\bR)$, and $[\omega]_{A_p}\leq K_0$. Suppose that \eqref{ellip}-\eqref{rembound} are satisfied.
Assume that $u\in C_c^\infty((-\infty,T]\times\bR^d_+)$ satisfies
\begin{equation*}
a_0 u_t - x_d^2 a_{ij} D_{ij}u + x_d b_i D_i u + cu + \lambda c_0 u = f
\end{equation*}
in $\Omega_T$.
Then there exists sufficiently small
\begin{equation*}
    \gamma_0=\gamma_0(d,p,\nu,K_0)>0
\end{equation*}
 such that if Assumption \ref{ass_lead} $(\rho_0,\gamma_0)$ is satisfied, then we have
\begin{eqnarray} \label{eq2121719}
  &\|u_t\|_{\bL_{p,\theta,\omega}(T)} + (1+\lambda)\|u\|_{\bL_{p,\theta,\omega}(T)}& \nonumber
  \\
  &+ (1+\sqrt{\lambda})\|x_d D_x u\|_{\bL_{p,\theta,\omega}(T)} + \|x_d^2 D_x^2 u\|_{\bL_{p,\theta,\omega}(T)}& \nonumber
  \\
  &\leq N (\|u\|_{\bL_{p,\theta,\omega}(T)} + \|f\|_{\bL_{p,\theta,\omega}(T)}),&
\end{eqnarray}
where $N=N(d,p,\theta,K,\nu,K_0)$.
\end{lemma}

\begin{proof}
For any function $h$ defined on $\Omega_T$, we write $h_r(t,x):=h(t,x/r)$.
  Let $\zeta\in C_c^\infty((2,3))$ be a standard nonnegative cutoff function.
Note that $v(t,x):=\zeta(x_d)u_r(t,x) \in W_{p,\omega}^{1,2}((-\infty,T)\times \bR^d)$, where
  \begin{equation*}
    W_{p,\omega}^{1,2}(T) = \{u: u,D_xu, D_x^2u, u_t \in L_{p}((-\infty,T)\times \bR^d,\omega(t)dtdx)\}.
  \end{equation*}
We see that $v$ satisfies
  \begin{equation} \label{eq2121809}
    a_{0,r}v_t-x_d^2 a_{ij,r}D_{ij}v+ (\Lambda+\lambda)c_{0,r} v=g
  \end{equation}
  in $(-\infty,T)\times\bR^d$, where $\Lambda\geq0$ and
  \begin{align} \label{eq8142335}
    g&:=\zeta f_r - x_d^2 a_{dd,r}\zeta''u_r-x_d^2\sum_{i\neq d} (a_{id,r}+a_{di,r})\zeta' D_iu_r \nonumber
    \\
    &\quad-x_d b_{i,r} \zeta D_iu_r - c_r\zeta u_r + \Lambda c_{0,r}\zeta u_r.
  \end{align}

Using \eqref{eq2111512}, one can show that
there is $R_0>0$ such that the (unweighted) oscillation of $x_d^2 a_{ij,r}$ on $Q_{R}(t,x):=(t-R^2,t)\times B_R(x)$ is less than $N\gamma_0$ if $x_d\in(1,4)$ and $R\leq R_0$.
By \cite[Theorem 6.3]{DK18}, there is $\Lambda_0=\Lambda_0(d,p,\nu,K_0)\geq0$ such that if $\Lambda=\Lambda_0$, then we can apply the $W_{p,\omega}^{1,2}$-estimate for the equation \eqref{eq2121809} to get
  \begin{eqnarray} \label{eq2121810}
    &\|v_t\|_{L_{p,\omega}(T)} + (\Lambda_0+\lambda) \|v\|_{L_{p,\omega}(T)} + \sqrt{\Lambda_0+\lambda} \|D_xv\|_{L_{p,\omega}(T)} + \|D^2_xv\|_{L_{p,\omega}(T)}& \nonumber
    \\
    &\leq N \|g\|_{L_{p,\omega}(T)},&
  \end{eqnarray}
  where $L_{p,\omega}(T):=L_{p}((-\infty,T)\times \bR^d,\omega(t)dtdx)$.
  Here, we remark that the arguments in \cite{DK18} still work for equations with general coefficients $a_0$ and $c_0$, and thus we can still utilize the estimate established in the paper (see also \cite[Remark 3.7]{DR24div}).
Note that for any $\gamma\in\bR$,
  \begin{align*}
    \int_0^\infty \left( \int_{\Omega_{T}} |D_d^j\zeta(x_d)h_r(t,x)|^p \omega(t) dxdt \right) r^\gamma dr &= N_j \int_{\Omega_T} |h(t,x)|^p \omega(t) x_d^{-\gamma-d-1} dxdt,
  \end{align*}
  where 
  \begin{equation*}
    N_j:=\int_0^\infty |D_s^i\zeta(s)|^p s^{\gamma+d} ds, \quad j=0,1,2.
  \end{equation*}
  By raising both sides of \eqref{eq2121810} to the power of $p$, multiply by $r^{-\theta-d}$, and integrating with respect to $r$ on $(0,\infty)$, we obtain that
\begin{equation} 
	\begin{aligned}
		  &\|u_t\|_{\bL_{p,\theta,\omega}(T)} + (1+\lambda)\|u\|_{\bL_{p,\theta,\omega}(T)}+ (1+\sqrt{\lambda})\|x_d D_x u\|_{\bL_{p,\theta,\omega}(T)} + \|x_d^2 D_x^2u\|_{\bL_{p,\theta,\omega}(T)}\\
		&\leq N \left( \|f\|_{\bL_{p,\theta,\omega}(T)} + \|u\|_{\bL_{p,\theta,\omega}(T)} + \|x_dD_xu\|_{\bL_{p,\theta,\omega}(T)} \right).
	\end{aligned}
	\label{eq2122053}
\end{equation}
Note that by the interpolation inequality in the whole space,
\begin{equation*}
  \|h\|_{W_p^1(\bR^d)} \leq N\|h\|_{W_p^2(\bR^d)}^{1/2} \|h\|_{L_p(\bR^d)}^{1/2}.
\end{equation*}
By this and the relation \eqref{equivnorm}, we obtain
\begin{equation} \label{eq2152354}
  \begin{aligned}
  	  &\|x_dD_xh\|_{L_{p,\theta}} \leq \|h\|_{H^1_{p,\theta}}\leq N\left( \sum_{m=-\infty}^\infty e^{m(\theta+d-1)}\|h(t,e^m\cdot)\zeta\|_{W_p^1(\bR^d)}^p \right)^{1/p}\\
  	&\leq N \left( \sum_{m=-\infty}^\infty e^{m(\theta+d-1)}\|h(t,e^m\cdot)\zeta\|_{W_p^2(\bR^d)}^{p/2} \|h(t,e^m\cdot)\zeta\|_{L_{p}(\bR^d)}^{p/2} \right)^{1/p}\\
  	&\leq N\|h\|_{H^2_{p,\theta}}^{1/2}\|h\|_{L_{p,\theta}}^{1/2},
  \end{aligned}
\end{equation}
which yields that
\begin{equation*}
  \|x_dD_xu\|_{\bL_{p,\theta,\omega}(T)} \leq \|u\|_{\bH^1_{p,\theta,\omega}(T)} \leq N\|u\|_{\bH^2_{p,\theta,\omega}(T)}^{1/2} \|u\|_{\bL_{p,\theta,\omega}(T)}^{1/2}.
\end{equation*}
This and \eqref{eq2122053} together with Young's inequality lead to the desired estimate \eqref{eq2121719}. 
The lemma is proved.
\end{proof}

Our next step is to prove Lemma \ref{lem2122217}, where equations with simple coefficients are considered. 
Before we present this, we give some definitions which are used in the proof of the lemma.

We present the definition of weak solutions to
\begin{equation} \label{eqdiv}
    a_0u_t-x_d^2D_i(a_{ij}D_ju) + x_db_i D_iu + cu + \lambda c_0u = f.
\end{equation}

\begin{definition}
  Let $p,q\in (1,\infty)$, $\theta\in\bR$, $T\in(-\infty,\infty]$, $\omega\in A_q(\bR)$, and $f\in \bL_{q,p,\theta,\omega}(T)$. 

  $(i)$ In the case when $a_0=a_0(x_d)$, we say that $u\in \bH^1_{q,p,\theta,\omega}(T)$ is a weak solution to \eqref{eqdiv} if
  \begin{align*}
    &-\int_{\Omega_T} a_0 u\varphi_t dxdt + \int_{\Omega_T} a_{ij} D_juD_{i}(x_d^2\varphi) dxdt + \int_{\Omega_T} x_d b_i \varphi D_iudxdt 
    \\
    &\quad + \int_{\Omega_T} c u\varphi dxdt + \int_{\Omega_T} \lambda c_0 u\varphi dxdt = \int_{\Omega_T} f\varphi dxdt
  \end{align*}
  for any $\varphi \in C_c^\infty((-\infty,T)\times \bR^d_+)$.

  $(ii)$ In the case when $a_0=a_0(t)$, we say that $u\in \bH^1_{q,p,\theta,\omega}(T)$ is a weak solution to \eqref{eqdiv} if
  \begin{align*}
    &-\int_{\Omega_T} u\varphi_t dxdt + \int_{\Omega_T} \frac{a_{ij}}{a_0} D_juD_{i}(x_d^2\varphi) dxdt + \int_{\Omega_T} x_d \frac{b_i}{a_0} \varphi D_iu dxdt 
    \\
    &\quad + \int_{\Omega_T} \frac{c}{a_0} u\varphi dxdt + \int_{\Omega_T} \lambda \frac{c_0}{a_0} u\varphi dxdt = \int_{\Omega_T} \frac{1}{a_0} f\varphi dxdt
  \end{align*}
  for any $\varphi \in C_c^\infty((-\infty,T)\times \bR^d_+)$.
\end{definition}

Next, we introduce two classes of simple coefficients, each aligned with the $[\cdot]_{\rho,x_0}$-type coefficients in Assumptions \ref{ass_lead} $(\rho_0,\gamma_0)$ and \ref{ass_coeff} $(\rho_0,\gamma_0)$, respectively.

\begin{assumption} \label{simple1}
\textbf{} 
   \begin{itemize}
    \item $a_0, a_{dd}$, and $c_0$ depend on the same single variable, either $x_d$ or $t$,

  \item $a_{ij}$ depend only on $(t,x_d)$ for $(i,j)\neq (d,d)$,
\end{itemize}
\end{assumption}

The following is a stronger assumption.

\begin{assumption} \label{simple2}
\textbf{} 
   \begin{itemize}
    \item $a_0, a_{dd}, b_d, c$, and $c_0$ depend on the same single variable, either $x_d$ or $t$,

  \item $a_{ij}$ depend only on $(t,x_d)$ for $(i,j)\neq (d,d)$,

  \item $b_i$ depend only on $(t,x_d)$ for $i\neq d$,

  \item we have
      \begin{equation} \label{eq2171340}
    \frac{b_d}{a_{dd}}=n_b, \quad \frac{c}{a_{dd}}=n_c
  \end{equation}
  for some $n_b, n_c\in \bR$.
\end{itemize}
\end{assumption}

\begin{lemma} \label{lem2122217}
      Let $T\in(-\infty,\infty]$, $\lambda\geq0$, $p,q\in(1,\infty)$, $\omega\in A_q(\bR)$, and $[\omega]_{A_q}\leq K_0$. Suppose that \eqref{ellip}-\eqref{rembound} are satisfied.

(i) There exists $\lambda_0 = \lambda_0(d,p,q,\theta,\nu,K,K_0)\geq0$ such that under Assumption \ref{simple1}, the following assertions hold true. For any $\lambda\geq\lambda_0$ and $f\in \bL_{q,p,\theta,\omega}(T)$, there is a unique solution $u\in \frH_{q,p,\theta,\omega}^2(T)$ to \eqref{mainpara}. Moreover, for this solution,
\begin{eqnarray} \label{eq2122338}
  &\|u_t\|_{\bL_{q,p,\theta,\omega}(T)} + (1+\lambda)\|u\|_{\bL_{q,p,\theta,\omega}(T)}& \nonumber
  \\
  &+ (1+\sqrt{\lambda})\|x_d D_x u\|_{\bL_{q,p,\theta,\omega}(T)} + \|x_d^2 D_x^2 u\|_{\bL_{q,p,\theta,\omega}(T)}& \nonumber
  \\
  &\leq N \|f\|_{\bL_{q,p,\theta,\omega}(T)}.&
\end{eqnarray}
where $N=N(d,q,p,\theta,\nu,K,K_0)$.

(ii)
Let $n_b,n_c\in\bR$, and suppose that the quadratic equation \eqref{eq2152256} has two distinct real roots $\alpha<\beta$.
Then under Assumption \ref{simple2}, the assertions in (i) hold with $\lambda_0=0$ and $\alpha p<\theta<\beta p$,
where the dependencies of $N$ are replaced by $d,q,p,\theta,n_b,n_c,\nu,K$, and $K_0$.
\end{lemma}

\begin{proof}
We present the proofs for $(i)$ and $(ii)$ together.

First, we prove the existence result for the case when $q=p$.
By dividing the equation \eqref{mainpara} by $a_{dd}$, and then absorbing $a_{dd}$ into $a_0$, we may assume that $a_{dd}=1$. Thus, the equation can be rewritten into a divergence form equation
   \begin{equation} \label{eqtrans}
     a_{0}u_t - x_d^2 D_{i}(\bar{a}_{ij} D_j u) + x_d b_i D_i u +cu +\lambda c_0 u= f,
   \end{equation}
   where
   \begin{equation*}
     \bar{a}_{ij}:=\left\{\begin{array}{ll}
       a_{ij}+a_{ji} &\text{ for } i\neq d \text{ and } j=d;
       \\
       0 &\text{ for } i= d \text{ and } j\neq d;
       \\
       a_{ij} &\text{ otherwise}.
     \end{array} \right.
   \end{equation*}
   Here, in the case when Assumption \ref{simple2} holds, the ratio condition \eqref{eq2171340} is still satisfied. This implies that Assumption $2.2$ $(\rho_0,\gamma_0)$ or $2.3$ $(\rho_0,\gamma_0)$ in \cite{DR24div} is satisfied.
  Hence, by \cite[Theorem 2.6]{DR24div}, there exists $\lambda_0\geq0$ such that for any $\lambda \geq \lambda_0$, one can find a solution $u\in \bH_{p,\theta,\omega}^1(T)$ to \eqref{eqtrans} satisfying
\begin{equation} \label{eq2122344}
(1+\sqrt{\lambda})\|u\|_{\bL_{p,\theta,\omega}(T)} + \|x_dD_xu\|_{\bL_{p,\theta,\omega}(T)} \leq \frac{N}{1+\sqrt{\lambda}} \|f\|_{\bL_{p,\theta,\omega}(T)}.
\end{equation}

Now we show that this solution $u$ is actually in $\frH_{p,\theta,\omega}^2(T)$.
 Let us consider $v(t,x):=\zeta(x_d) u_r(t,x) := \zeta(x_d) u(t,x/r)$ where $\zeta\in C_c^\infty((2,3))$ is a standard nonnegative cutoff function. Then we observe that $v$ satisfies
    \begin{align} \label{eq2122354}
    &a_{0,r}v_t-x_d^2 a_{ij,r}D_{ij}v+ (\Lambda +\lambda) c_{0,r} v \nonumber
    \\
    &= a_{0,r}v_t-x_d^2 D_i(\bar{a}_{ij,r}D_{j}v)+ (\Lambda +\lambda) c_{0,r} v = g,
  \end{align}
   where $\Lambda\geq0$, and $g$ is defined as \eqref{eq8142335}.
 Here, $h_r(t,x):=h(t,x/r)$ for any function $h$ on $\Omega_T$. In particular, $v \in \cH_{p,\omega}^1(T)$ where $\cH_{p,\omega}^1(T)$ is the set of functions such that
 $v,D_xv \in L_{p}((-\infty,T)\times \bR^d,\omega(t)dtdx)$ and $a_{0,r}h_t \in L_{p}((-\infty,T),\omega dt; H_p^{-1}(\bR^d))$.
Since $g\in L_{p}((-\infty,T)\times \bR^d,\omega(t)dtdx)$, by \cite[Theorem 6.2]{DK18}, we have a solution to \eqref{eq2122354} in the space $W_{p,\omega}^{1,2}(T)$, say $\tilde{v}$. Since $\tilde{v}$ is also in the space $\cH_{p,\omega}^1(T)$, the uniqueness result in $\cH_{p,\omega}^1(T)$ (see also \cite[Remark 3.7]{DR24div}) yields that $v=\tilde{v}\in W_{p,\omega}^{1,2}(T)$, and
\begin{eqnarray} \label{eq2152234}
    &\|v_t\|_{L_{p,\omega}(\Omega_T)} + (1+\lambda)\|v\|_{L_{p,\omega}(\Omega_T)}+ (1+\sqrt{\lambda})\|D_x v\|_{L_{p,\omega}(\Omega_T)} + \| D_x^2 v\|_{L_{p,\omega}(\Omega_T)}& \nonumber
  \\
  &\leq N \|g\|_{L_{p,\omega}(\Omega_T)}.&
\end{eqnarray}
As in \eqref{eq2122053}, we raise both sides of \eqref{eq2152234} to the power of $p$, multiply by $r^{-\theta-d}$, and integrate with respect to $r$. Then by \eqref{eq2122344},
\begin{align*}
  &\|u_t\|_{\bL_{p,\theta,\omega}(T)} + (1+\lambda)\|u\|_{\bL_{p,\theta,\omega}(T)}+ (1+\sqrt{\lambda})\|x_d D_x u\|_{\bL_{p,\theta,\omega}(T)} + \|x_d^2 D_x^2 u\|_{\bL_{p,\theta,\omega}(T)}
  \\
  &\leq N \left(\|f\|_{\bL_{p,\theta,\omega}(T)} + \|u\|_{\bL_{p,\theta,\omega}(T)} + \|x_dD_xu\|_{\bL_{p,\theta,\omega}(T)}\right) \leq N\|f\|_{\bL_{p,\theta,\omega}(T)}.
\end{align*}
Thus, $u\in \frH_{p,\theta,\omega}^2(T)$, and it satisfies \eqref{eq2122338}.

Before considering the general case $q\neq p$, we remark that $u$ is independent of $p,\theta$, and $\omega$ when $f\in C_c^\infty(\Omega_T)$. In other words, for $p_i\in(1,\infty)$, $\omega_i\in A_{p_i}(\bR)$, and $\theta_i\in \bR$ ($\theta_i \in (\alpha p_i,\beta p_i)$ for the case $(ii)$), if $u_i \in \frH_{p_i,\theta_i,\omega_i}^2(T)$ are two solutions to \eqref{mainpara} with $f\in C_c^\infty(\Omega_T)$, then $u_1=u_2$. This fact easily follows from \cite[Remark 4.6]{DR24div}, where the independence for divergence form equation \eqref{eqtrans} is addressed.

Next, we deal with the existence when $q\neq p$.
Let $f\in C_c^\infty(\Omega_T)$, $\omega'\in A_p(\bR)$, and take a solution $u\in \frH_{p,\theta,\omega'}^2(T)$. Since $u$ is independent of $\omega'$, it satisfies the estimate \eqref{eq2122338} with $q=p$ and any $\omega' \in A_p(\bR)$. This and the extrapolation theorem (see e.g. \cite[Theorem 2.5]{DK18}) yield that $u\in \frH_{q,p,\theta,\omega}^2(T)$, and \eqref{eq2122338} for general $q\neq p$ and $\omega \in A_q(\bR)$. 

Lastly, we consider the uniqueness result. Let $u \in \frH_{q,p,\theta,\omega}^2(T)$ be a solution to \eqref{mainpara} with $f=0$. 
As in the first case $p=q$, we may assume that $a_{dd}=1$, and $u$ satisfies the divergence form equation \eqref{eqtrans}. By the uniqueness result from \cite[Theorem 2.6]{DR24div}, we have $u=0$.
The lemma is proved.
\end{proof}

\begin{lemma} \label{lem2131452}
  Let $T\in(-\infty,\infty]$, $\rho_0\in(1/2,1)$, $\gamma_0>0$, $p\in(1,\infty)$, $\omega\in A_q(\bR)$, and $[\omega]_{A_q}\leq K_0$. Suppose that \eqref{ellip}-\eqref{rembound} are satisfied.
  Assume that $u\in \frH_{p,\theta,\omega}^2(T)$ satisfies \eqref{mainpara} where $f\in \bL_{p,\theta,\omega}(T)$, and $u$ and $f$ are compactly supported on $(-\infty,T]\times B_{\rho_0}(x_0)$ for some $x_0\in \bR^d_+$ with $x_{0d}=1$.

  (i)
 Let $p_1 \in (1,p)$. Then there exists
 \begin{equation*}
\lambda_0=\lambda_0(d,p,p_1,\theta,\nu,K,K_0)\geq0
 \end{equation*}
 such that under Assumption \ref{ass_lead} $(\rho_0,\gamma_0)$, the following assertion holds true.
  For any $\lambda\geq\lambda_0$ and $\varepsilon\in(0,1)$,
  \begin{align} \label{eq2131519}
    \|u\|_{\bL_{p,\theta,\omega}(T)}  &\leq (\varepsilon + N_{\rho_0,\varepsilon} \gamma_0^{(p-p_1)/pp_1})(\lambda\|u\|_{\bL_{p,\theta,\omega}(T)} + \|u\|_{\frH_{p,\theta,\omega}^2(T)}) \nonumber
    \\
    &\quad+ N \|f\|_{\bL_{p,\theta,\omega}(T)},
  \end{align}
    where $N$ depends only on $d,p,p_1,\theta,\nu,K$, and $K_1$, and $N_{\rho_0,\varepsilon}$ depends only on $d,p,p_1,\theta,\nu,K,K_0,\rho_0$, and $\varepsilon$.

    (ii) 
Let $n_b,n_c\in\bR$ and the quadratic equation \eqref{eq2152256} has two distinct real roots $\alpha$ and $\beta$. Let $p_1\in(1,p)$ and $\theta\in(\alpha p,\beta p)$ such that $\theta\in(\alpha p_1,\beta p_1)$.
   Then under Assumption \ref{ass_coeff} $(\rho_0,\gamma_0)$, the assertion in $(i)$ holds with $\lambda_0=0$, where the constant $N$ depends only on $d,p,p_1,\theta,n_b,n_c,\nu,K$, and $K_0$, and $N_{\rho_0,\varepsilon}$ depends only on $d,p,p_1,\theta,n_b,n_c,\nu,K,K_0,\rho_0$, and $\varepsilon$.
\end{lemma}

\begin{proof}
As in the proof of Lemma \ref{lem2122217}, we prove $(i)$ and $(ii)$ together.

Since Assumption \ref{ass_lead} $(\rho_0,\gamma_0)$ or \ref{ass_coeff} $(\rho_0,\gamma_0)$ holds, we can take the coefficients $[a_{0}]_{\rho_0,x_0}, [a_{ij}]_{\rho_0,x_0}, [b_{i}]_{\rho_0,x_0}, [c]_{\rho_0,x_0}$, and $[c_0]_{\rho_0,x_0}$ satisfying Assumption \ref{simple1} or \ref{simple2}. In particular, when Assumption \ref{ass_lead} $(\rho_0,\gamma_0)$ is satisfied, we put $[b_{i}]_{\rho_0,x_0}=b_i$ and $[c]_{\rho_0,x_0}=c$.
  By Lemma \ref{lem2122217}, we can take $\lambda_0\geq0$ such that for any  $\lambda\geq\lambda_0$, there is a solution $v\in \frH_{p,\theta,\omega}^2(T)$ to
  \begin{equation*}
    \sL_0v=f,
  \end{equation*}
 where
  \begin{equation*}
    \sL_0 v := [a_{0}]_{\rho_0,x_0} v_t - x_d^2 [a_{ij}]_{\rho_0,x_0}D_{ij}v + x_d [b_{i}]_{\rho_0,x_0}D_iv + [c]_{\rho_0,x_0}v +\lambda [c_{0}]_{\rho_0,x_0} v.
  \end{equation*}
  Due to \eqref{eq2122338}, we also have
  \begin{equation} \label{eq2131649}
    \|v\|_{\frH_{p,\theta,\omega}^2(T)} \leq N  \|f\|_{\bL_{p,\theta,\omega}(T)}.
  \end{equation}
   Note that $f\in \bL_{p,p_1,\theta,\omega}(T)$ since $f$ is compactly supported on $(-\infty,T]\times B_{\rho_0}(x_{0})$. Thus, as described in the proof of Lemma \ref{lem2122217}, $v$ is also in the space $\frH_{p,p_1,\theta,\omega}^2(T)$.

     Now we consider $w:=u-v$, which satisfies
  \begin{align*}
    \sL_0w &= ([a_{0}]_{\rho_0,x_0}-a_0) u_t + x_d^2 (a_{ij}-[a_{ij}]_{\rho_0,x_0})D_{ij}u - x_d (b_i - [b_{i}]_{\rho_0,x_0}) D_iu 
    \\
    &\quad - (c-[c]_{\rho_0,x_0})u -\lambda (c_0-[c_0]_{\rho_0,x_0})u.
  \end{align*}
By \eqref{eq2122338} with $(p,p_1)$ instead of $(q,p)$, we obtain
\begin{align} \label{eq2131728}
    \|w\|_{\bL_{p,p_1,\theta,\omega}(T)} &\leq N \|\sL_0w\|_{\bL_{p,p_1,\theta,\omega}(T)} \nonumber
    \\
    &\leq N\|([a_{0}]_{\rho_0,x_0}-a_0) u_t\|_{\bL_{p,p_1,\theta,\omega}(T)} \nonumber
    \\
    &\quad + N \|(a_{ij}-[a_{ij}]_{\rho_0,x_0})x_dD_{x}u\|_{\bL_{p,p_1,\theta,\omega}(T)} \nonumber
  \\
  &\quad + N \|(b_{i}-[b_{i}]_{\rho_0,x_0})x_dD_{x}u\|_{\bL_{p,p_1,\theta,\omega}(T)} \nonumber
  \\
  &\quad + N \|(c-[c]_{\rho_0,x_0}) u\|_{\bL_{p,p_1,\theta,\omega}(T)} \nonumber
  \\
  &\quad + N\lambda \|(c_0-[c_0]_{\rho_0,x_0}) u\|_{\bL_{p,p_1,\theta,\omega}(T)}.
\end{align}
Since $supp \, u(t,\cdot)\subset B_{\rho_0}(x_{0d})$, by H\"older's inequality and Assumption \ref{ass_lead} $(\rho_0,\gamma_0)$ or Assumption \ref{ass_coeff} $(\rho_0,\gamma_0)$,
\begin{align*}
  &\|([a_{0}]_{\rho_0,x_0}-a_0) u_t\|_{\bL_{p,p_1,\theta,\omega}(T)} \nonumber
  \\
  &\leq \left(\sup_{t\leq T}\int_{B_{\rho_0}(x_0)} |[a_{0}]_{\rho_0,x_0}-a_0|^{p_1q} x_d^{\theta-1} dx\right)^{1/p_1q} \|u_t\|_{\bL_{p,\theta,\omega}(T)} \nonumber
  \\
  &\leq N_{\rho_0} \left( \sup_{t\leq T} \aint_{B_{\rho_0}(x_0)} |[a_{0}]_{\rho_0,x_0}-a_0| dx \right)^{1/p_1q} \|u_t\|_{\bL_{p,\theta,\omega}(T)} \nonumber
  \\
  &\leq N_{\rho_0} \gamma_0^{1/p_1q}\|u_t\|_{\bL_{p,\theta,\omega}(T)},
\end{align*}
where $q:=p/(p-p_1)$ is the H\"older conjugate of $p/p_1$. Similarly, we can bound the last four terms in \eqref{eq2131728} (recall that $[b_{i}]_{\rho_0,x_0}=b_i$ and $[c]_{\rho_0,x_0}=c$ if Assumption \ref{ass_lead} $(\rho_0,\gamma_0)$ holds) to obtain that
\begin{align} \label{eq2141333}
  \|w\|_{\bL_{p,p_1,\theta,\omega}(T)} \leq N_{\rho_0} \gamma_0^{1/p_1q}(\lambda\|u\|_{\bL_{p,\theta,\omega}(T)} + \|u\|_{\frH^2_{p,\theta,\omega}(T)}).
\end{align}
Next, we estimate $\bL_{p,\theta,\omega}(T)$-norm of $w$. By the (unweighted) Gagliardo-Nirenberg interpolation inequality in $x$-variable,
\begin{align*}
  \|h\|_{L_{p}(\bR^d)} &\leq N\|D_x^2h\|_{L_{p}(\bR^d)}^\kappa \|h\|_{L_{p_1}(\bR^d)}^{1-\kappa},
\end{align*}
where 
\begin{equation*}
  \frac{1}{p}=\kappa\left(\frac{1}{p}-\frac{2}{d}\right)+(1-\kappa)\frac{1}{p_1}, \quad \kappa\in(0,1).
\end{equation*}
From this and \eqref{equivnorm}, for $\zeta \in C_c^\infty(\bR_+)$ satisfying \eqref{eq2102340},
\begin{align*}
  &\|w\|_{\bL_{p,\theta,\omega}(T)} \nonumber
  \\
  &= \left( \int_{-\infty}^T \sum_{m=-\infty}^\infty e^{m(\theta+d-1)}\|w(t,e^m\cdot)\zeta\|_{L_p(\bR^d)}^p \omega(t) dt \right)^{1/p} \nonumber
  \\
  &\leq N \left( \int_{-\infty}^T \sum_{m=-\infty}^\infty e^{m(\theta+d-1)}\|w(t,e^m\cdot)\zeta\|_{W_p^2(\bR^d)}^{\kappa p} \|w(t,e^m\cdot)\zeta\|_{L_{p_1}(\bR^d)}^{(1-\kappa) p} \omega(t) dt \right)^{1/p} \nonumber
  \\
  &\leq N\|w\|_{\bH^2_{p,\theta,\omega}(T)}^{\kappa}\|w\|_{\bL_{p,p_1,\theta,\omega}(T)}^{1-\kappa}.
\end{align*}
Hence, by \eqref{eq2141333}, for any $\varepsilon\in(0,1)$,
\begin{align*}
  \|w\|_{\bL_{p,\theta,\omega}(T)} &\leq  \varepsilon \|w\|_{\bH_{p,\theta,\omega}^2(T)} + N_{\varepsilon} \|w\|_{\bL_{p,p_1,\theta,\omega}(T)} \nonumber
  \\
  &\leq  \varepsilon \|w\|_{\bH_{p,\theta,\omega}^2(T)} + N_{\rho_0, \varepsilon} \gamma_0^{1/p_1q}(\lambda\|u\|_{\bL_{p,\theta,\omega}(T)} + \|u\|_{\frH^2_{p,\theta,\omega}(T)}) \nonumber
  \\
  &\leq (\varepsilon + N_{\rho_0,\varepsilon} \gamma_0^{1/p_1q})(\lambda \|u\|_{\bL_{p,\theta,\omega}(T)} + \|u\|_{\frH^2_{p,\theta,\omega}(T)}) + \varepsilon \|v\|_{\frH_{p,\theta,\omega}^2(T)}.
\end{align*}
Since $u=v+w$, this and \eqref{eq2131649} lead to the desired estimate \eqref{eq2131519}.
The lemma is proved.
\end{proof}

Now we prove our first main result.

\begin{proof}[Proof of Theorem \ref{thmpara}]
Due to Lemma \ref{lem2122217}, the method of continuity, and the extrapolation theorem, it suffices to prove \eqref{eq2111600} with $p=q$. Thanks to the denseness of $\frH_{q,p,\theta,\omega}^2(T)$, we prove the estimate when $u\in C_c^\infty((-\infty,T]\times\bR^d_+)$.
Let $\lambda_0$ taken from Lemma \ref{lem2131452} and denote
\begin{equation*}
  f:= \sL_p u + \lambda c_0u.
\end{equation*}
  Let $\varepsilon_0\in(0,1)$. We consider a partition of unity with weight. By \cite[Lemma 5.6]{KL13}, there exists $\rho_0=\rho_0(\varepsilon_0)\in(1/2,1)$, and nonnegative $\eta_k\in C_c^\infty(\bR^{d}_+)$ such that
  \begin{eqnarray} \label{eq2141547}
    \sum_{k=1}^\infty \eta_k^p\geq1, \quad \sum_{k=1}^\infty \eta_k \leq N(d), \quad \sum_{k=1}^\infty\left( x_d|D_x\eta_k|+ x_d^2|D_x^2\eta_k| \right) \leq \varepsilon_0^p,
  \end{eqnarray}
  Moreover, for each $k$, $supp\,\eta_k \subset B_{\rho_0 x_{kd}}(x_k)$ for some $x_k\in \bR^{d}_+$.

 Then $u_k:=u\eta_k$ satisfies
   \begin{align*}
     \sL_p u_k + \lambda c_0 u_k = f\eta_k-x_d^2(a_{ij}+a_{ji})D_i u D_j \eta_k -x_d^2 a_{ij} u D_{ij} \eta_k + x_d b_i u_k D_i\eta_k.
   \end{align*}
   Now we apply Lemma \ref{lem2131452} to $v(t,x)=u_k(t,x_{kd}x)$. Take $p_1\in(1,p)$ such that $\alpha p_1<\theta<\beta p_1$. Then for any $\varepsilon\in(0,1)$,
\begin{align*}
  & \|u_k\|_{\bL_{p,\theta,\omega}(T)} 
  \\
  &\leq (\varepsilon + N_{\rho_0,\varepsilon} \gamma_0^{(p-p_1)/pp_1})(\lambda \|u_k\|_{\bL_{p,\theta,\omega}(T)} + \|u_k\|_{\frH_{p,\theta,\omega}^2(T)}) + N \|f\eta_k\|_{\bL_{p,\theta,\omega}(T)}
  \\
  &\quad+ N \|x_d^2 D_xu D_x\eta_k\|_{\bL_{p,\theta,\omega}(T)} + N \|x_d^2 u D_x^2 \eta_k\|_{\bL_{p,\theta,\omega}(T)} + N \| x_d u_k D_i\eta_k \|_{\bL_{p,\theta,\omega}(T)}.
\end{align*}
By raising both sides of this inequality to the power of $p$, and summing in $k$, \eqref{eq2141547} yields that
\begin{align*}
  \|u\|_{\bL_{p,\theta,\omega}(T)} &\leq (N\varepsilon + N_{\rho_0,\varepsilon} \gamma_0^{(p-p_1)/pp_1}) (\lambda\|u\|_{\bL_{p,\theta,\omega}(T)} + \|u\|_{\frH_{p,\theta,\omega}^2(T)})
  \\
  &\quad + N \|f\|_{\bL_{p,\theta,\omega}(T)} + N\varepsilon_0\left( \| u\|_{\bL_{p,\theta,\omega}(T)}+\| x_d D_x u\|_{\bL_{p,\theta,\omega}(T)} \right).
\end{align*}
This and higher-order estimate \eqref{eq2121719} yield
\begin{align*}
  &\|u_t\|_{\bL_{p,\theta,\omega}(T)} + (1+\lambda)\|u\|_{\bL_{p,\theta,\omega}(T)} + (1+\sqrt{\lambda})\|x_dD_xu\|_{\bL_{p,\theta,\omega}(T)} + \|x_d^2D_x^2u\|_{\bL_{p,\theta,\omega}(T)} 
  \\
  &\leq  (N\varepsilon + N_{\rho_0,\varepsilon} \gamma_0^{(p-p_1)/pp_1}) (\lambda\|u\|_{\bL_{p,\theta,\omega}(T)} + \|u\|_{\frH_{p,\theta,\omega}^2(T)})
  \\
  &\quad + N\|f\|_{\bL_{p,\theta,\omega}(T)} + N\varepsilon_0\left( \| u\|_{\bL_{p,\theta,\omega}(T)} + \| x_d D_x u\|_{\bL_{p,\theta,\omega}(T)} \right).
\end{align*}
We choose $\varepsilon_0$ sufficiently small so that $N\varepsilon_0 <1/3$. Then $\rho_0=\rho_0(\varepsilon_0)$ is determined by the choice of $\eta_k$. Next, we take a small $\varepsilon$, and choose sufficiently small $\gamma_0$ so that $N\varepsilon + N_{\varepsilon} \gamma_0^{(p-p_1)/pp_1}<1/3$. Thus, we obtain the estimate \eqref{eq2111600} when $q=p$. The theorem is proved.
\end{proof}

To prove Theorem \ref{thmfinite}, it suffices to repeat the proof of \cite[Theorem 2.9]{DR24div}, where we adopt the idea of the proof of \cite[Theorem 2.1]{KryVMO}. For the reader's convenience, we provide the full proof.

\begin{proof}[Proof of Theorem \ref{thmfinite}]

We first prove the a priori estimate \eqref{eq2121343} as well as the uniqueness. 
It is enough to consider $u\in C_c^\infty([0,T]\times\bR^d)$ with $u(0,\cdot)=0$ due to the definition of $\mathring{\frH}_{q,p,\theta,\omega}^2(0,T)$. Let us extend $u$ and $f$ to be zero for $t\leq0$. Then by zero initial condition, $u\in \frH_{q,p,\theta,\omega}^2(T)$ and it satisfies \eqref{mainfinite} in $(-\infty,T)\times\bR^d_+$.
For any $\bar{\lambda}\geq0$, we set $v:=e^{-\bar{\lambda}t}u \in \frH_{q,p,\theta,\omega}^2(T)$, which satisfies 
  \begin{equation*}
   \sL_p v +(\lambda c_0+\bar{\lambda}a_0)v = e^{-\bar{\lambda}t}f \text{ in } \Omega_T.
 \end{equation*}
Here, one can check that $(a_{ij})$, and $\frac{\lambda c_0+\bar{\lambda}a_0}{\lambda+\bar{\lambda}}$ with $[\frac{\lambda c_0+\bar{\lambda}a_0}{\lambda+\bar{\lambda}}]_{\rho,x_0}:=\frac{\lambda [c_0]_{\rho,x_0}+\bar{\lambda}[a_0]_{\rho,x_0}}{\lambda+\bar{\lambda}}$ satisfy Assumption \ref{ass_lead} $(\rho_0,\gamma_0)$. Thus, by applying Theorem \ref{thmpara} $(i)$, we can find $\rho_0\in(1/2,1)$ sufficiently close to $1$, a sufficiently small number $\gamma_0>0$,
 and a sufficiently large number $\lambda_0\geq0$ such that for any $\bar{\lambda}\geq\lambda_0$, we have
\begin{equation*}
\begin{aligned}
	  &\|v_t\|_{\bL_{q,p,\theta,\omega}(T)} + (1+\lambda+\bar{\lambda})\|v\|_{\bL_{q,p,\theta,\omega}(T)}\\
	&\quad + (1+\sqrt{\lambda+\bar{\lambda}}) \|x_dD_xv\|_{\bL_{q,p,\theta,\omega}(T)} + \|x_d^2D_x^2v\|_{\bL_{q,p,\theta,\omega}(T)}\\
	&\leq N \|e^{-\bar{\lambda}t}f\|_{\bL_{q,p,\theta,\omega}(T)},
\end{aligned}
\end{equation*}
 where $N$ is independent of $T$.
By taking $\bar{\lambda}=\lambda_0$,
 \begin{align*}
 &\|u_t\|_{\bL_{q,p,\theta,\omega}(0,T)} + (1+\lambda)\|u\|_{\bL_{q,p,\theta,\omega}(0,T)} 
 \\
 &\quad+ (1+\sqrt{\lambda}) \|x_dD_xu\|_{\bL_{q,p,\theta,\omega}(0,T)} + \|x_d^2D_x^2u\|_{\bL_{q,p,\theta,\omega}(0,T)} \nonumber
     \\
 &\leq N(T) \Big(\|v_t\|_{\bL_{q,p,\theta,\omega}(T)} + (1+\lambda+\lambda_0)\|v\|_{\bL_{q,p,\theta,\omega}(T)} 
 \\
 &\quad+ (1+\sqrt{\lambda+\lambda_0}) \|x_dD_xv\|_{\bL_{q,p,\theta,\omega}(T)} + \|x_d^2D_x^2v\|_{\bL_{q,p,\theta,\omega}(T)} \Big) \nonumber
  \\
  &\leq N(T) \|e^{-\bar{\lambda}t}f\|_{\bL_{q,p,\theta,\omega}(T)} \leq N(T) \|f\|_{\bL_{q,p,\theta,\omega}(T)}.
 \end{align*}
 Hence, \eqref{eq2121343} is proved.

Finally, we prove the existence result.
Note that the constant $N$ in the a priori estimate \eqref{eq2121343} is independent of $\lambda$. Thus, due to the method of continuity, it suffices to prove the existence for a single $\lambda\geq0$. Moreover, the estimate also allows us to consider the case $f\in C_c^\infty((0,T)\times\bR^d_+)$.
 We extend $f$ by zero for $t\leq0$. By Theorem \ref{thmpara} $(i)$, there are constants $\rho_0\in(1/2,1)$, $\gamma_0>0$, and $\lambda_0\geq0$ such that under Assumption \ref{ass_lead} $(\rho_0,\gamma_0)$, for any $\lambda\geq\lambda_0$, there is a solution $u\in \frH_{q,p,\theta,\omega}^2(T)$ to
\begin{equation*}
  \sL_p u+\lambda c_0 u = f \text{ in } \Omega_T.
\end{equation*}
 Let us fix such a $\lambda\geq\lambda_0$ and the corresponding solution $u$.
Since $f\in C_c^\infty((0,T)\times\bR^d_+)$, we take $\delta>0$ so that $f(t,\cdot)=0$ if $t<\delta$.
Thus, \eqref{eq2111600} with $T=\delta$ yields that $u(t,\cdot)=0$ for $t<\delta$. 
Hence, one can find a sequence of functions $u_n\in C_c^\infty((-\infty,T]\times\bR^d_+)$ such that $u_n(0,\cdot)=0$, and $u_n\to u$ in $\frH_{q,p,\theta,\omega}^2(T)$, which means that $u\in \mathring{\frH}_{q,p,\theta,\omega}^2(T)$ is a solution to \eqref{mainfinite} in $(0,T)\times\bR^d_+$. The theorem is proved.
\end{proof}

\section{Elliptic equations} \label{sec_ell}

In this section, we focus on the elliptic equations.
Our first lemma in this section is a version of Lemma \ref{lem2171542}, which is about higher-order regularity of solution.

\begin{lemma} \label{lem8150107}
Let $\rho_0\in(1/2,1)$, $\lambda\geq0$, $p\in(1,\infty)$, and $\theta\in \bR$. Suppose that \eqref{ellip}-\eqref{rembound} are satisfied.
Assume that $u\in H_{p,\theta}^2$ satisfies
\begin{equation*}
- x_d^2 a_{ij} D_{ij}u + x_d b_i D_i u + cu + \lambda c_0 u = f
\end{equation*}
in $\bR^d_+$.
Then there exists sufficiently small
\begin{equation*}
    \gamma_0=\gamma_0(d,p,\nu)>0
\end{equation*}
 such that if Assumption \ref{ass_lead} $(\rho_0,\gamma_0)$ is satisfied, then we have
\begin{eqnarray} \label{eq2171556}
  \|u\|_{L_{p,\theta}} + \|x D_x u\|_{L_{p,\theta}} + \|x^2 D_x^2 u\|_{L_{p,\theta}} \leq N \left( \|f\|_{L_{p,\theta}} + \|u\|_{L_{p,\theta}} \right),
\end{eqnarray}
where $N=N(d,p,\theta,K,\nu)$.
\end{lemma}

\begin{proof}
    
 Let $h_r(x):=h(x/r)$ for any function $h$ on $\bR^d_+$ and $v(x):=\zeta(x_d) u_r(x) := \zeta(x_d) u(x/r)$ where $\zeta\in C_c^\infty((2,3))$ is a standard nonnegative cutoff function. Then $v \in W_{p}^2(\bR^d)$ satisfies
    \begin{equation} \label{eq8151424}
    -x_d^2 a_{ij,r}D_{ij}v+ (\Lambda +\lambda) c_{0,r} v = g,
  \end{equation}
   where $\Lambda\geq0$ and
      \begin{equation} \label{eq2171558}
    g:=\zeta f_r - x_d^2 a_{dd,r}\zeta''u_r-x_d^2\sum_{i\neq d} (a_{id,r}+a_{di,r})\zeta' D_iu_r -x_d b_{i,r} \zeta D_iu_r + \Lambda c_{0,r}\zeta u_r - c_r\zeta u_r.
  \end{equation}
Now we apply \cite[Theorem 2.5]{D12} to obtain that
\begin{eqnarray} \label{eq2171550}
    (1+\lambda)\|v\|_{L_{p}(\bR^d)}+ (1+\sqrt{\lambda})\|D_x v\|_{L_{p}(\bR^d)} + \| D_x^2 v\|_{L_{p}(\bR^d)} \leq N \|g\|_{L_{p}(\bR^d)}.
\end{eqnarray}
Similar to \eqref{eq2122053}, we raise both sides of \eqref{eq2171550} to the power of $p$, multiply by $r^{-\theta-d}$, and integrate with respect to $r$. Then
\begin{align*}
  &(1+\lambda)\|u\|_{L_{p,\theta}} + (1+\sqrt{\lambda})\|x_dD_xu\|_{L_{p,\theta}} + \|x_d^2D_x^2u\|_{L_{p,\theta}}
  \\
  &\leq N \left(\|f\|_{L_{p,\theta}} + \|u\|_{L_{p,\theta}} + \|x_dD_xu\|_{L_{p,\theta}}\right).
\end{align*}
It remains to apply \eqref{eq2152354} and obtain $u\in H_{p,\theta}^2$. The lemma is proved.
\end{proof}

For the proof of Theorem \ref{thmell} $(iii)$, we handle the case when $d=1$. In this case, we denote the coefficients to be $a$ and $b$, instead of $(a_{ij})$ and $(b_i)$, respectively.

\begin{lemma} \label{lem2152133}
Let $p\in(1,\infty)$. Suppose that $a$ is a measurable function satisfying \eqref{ellip}, and $a, b$, and $c$ satisfy the ratio condition
      \begin{equation} \label{eq2181539}
    \frac{b}{a}=n_b, \quad \frac{c}{a}=n_c
  \end{equation}
  for some $n_b, n_c\in \bR$.
  Assume that the quadratic equation \eqref{eq2152256} has two distinct real roots $\alpha <\beta$.
Then for any $\theta\in \bR\setminus [\alpha p, \beta p]$ and $f \in L_{p,\theta}$, there is a unique solution $u\in H_{p,\theta}^2$ to
 \begin{equation} \label{eq2152047}
-x^2 aD_{x}^2u + x bD_xu  +cu = f.
\end{equation}
Moreover, for this solution, we have
\begin{eqnarray} \label{eq2152046}
  \|u\|_{L_{p,\theta}} + \|x D_x u\|_{L_{p,\theta}} + \|x^2 D_x^2 u\|_{L_{p,\theta}} \leq N \|f\|_{L_{p,\theta}},
\end{eqnarray}
where $N=N(p,\theta,n_b,n_c,\nu)$.
\end{lemma}

\begin{proof}
By dividing \eqref{eq2152047} by $a$, one may assume that $a, b$, and $c$ are constants.

 Due to the denseness of $C_c^\infty(\bR_+)$ in $L_{p,\theta}$,
we just need to show that for $f \in C_c^\infty(\bR_+)$, one can find a solution $u$ satisfying \eqref{eq2152046}.
  Since \eqref{eq2152047} is an ordinary differential equation, the general solution is given by
\begin{equation} \label{eq2152056}
  u(x)=(A_1(x)+B_1)x^{-\alpha} + (A_2(x)+B_2)x^{-\beta},
\end{equation}
where $B_1$ and $B_2$ are arbitrary constants,
\begin{align*}
  A_1(x) &:= -\frac{1}{a(\beta-\alpha)} \int_0^x y^{\alpha-1} f(y) \, dy,
\end{align*}
and
\begin{align*}
  A_2(x) &:= \frac{1}{a(\beta-\alpha)} \int_0^x y^{\beta-1} f(y) \, dy
\end{align*}
(see e.g. \cite[Theorem 3.6.1]{BDM21}). 

When $\theta<\alpha p$, if we choose $B_1=B_2=0$ in \eqref{eq2152056}, then  Hardy’s inequality yields that (see e.g. Theorem 5.1 in the preface of \cite{K85}) 
\begin{align*}
  \|u\|_{L_{p,\theta}} \leq \|A_1\|_{L_{p,\theta-\alpha p}} + \|A_2\|_{L_{p,\theta-\beta p}} \leq N\|f\|_{L_{p,\theta}}.
\end{align*}
Since $u$ defined by \eqref{eq2152056} is in the space $H_{p,\theta}^2$, one can apply \eqref{eq2171556} to obtain \eqref{eq2152046}.
For the case $\theta>\beta p$, one can still obtain the desired result by repeating the above argument with
\begin{align*}
  B_1 = \frac{1}{a(\beta-\alpha)} \int_0^\infty f(y) y^{\alpha-1} dy,
\end{align*}
and
\begin{align*}
  B_2 = \frac{1}{a(\beta-\alpha)} \int_0^\infty f(y) y^{\beta-1}dy.
\end{align*}

Next we prove the uniqueness of solution. Suppose that $u\in H_{p,\theta}^2$ is a solution to \eqref{eq2152047} with $f=0$. 
Let $v:=x^\alpha u$. Then by \eqref{eq2152256}, $v$ satisfies
\begin{align} \label{eq2152249}
    &-x^2aD_x^2v + x (2a\alpha+b) D_xv + (-a\alpha^2 -(a+b)\alpha +c)v \nonumber
    \\
    &= -x^2aD_x^2v + x (2a\alpha+b) D_xv = x^\alpha f.
\end{align}
This implies that we may assume that $c=0$.
In this case,
\begin{equation*}
  ax^2D_x^2u - bxD_xu = 0.
\end{equation*}
Then, by the standard formula for homogeneous Euler equation (ordinary differential equation),
\begin{equation*}
  u(x)=C_1x^{-\alpha} + C_2x^{-\beta}
\end{equation*}
for some constants $C_1,C_2\in\bR$. Here, we need to choose $C_1=C_2=0$ since $u\in H_{p,\theta}^2$. Thus, we have $u=0$.
The lemma is proved.
\end{proof}

\vspace{1em}

\begin{proof}[Proof of Theorem \ref{thmell}]

We first consider the case $(iii)$. Note that the case when $a, b$, and $c$ satisfy the ratio condition \eqref{eq2181539} is handled in Lemma \ref{lem2152133}. For general coefficients, we need to repeat the proofs of Lemma \ref{lem2131452} and Theorem \ref{thmpara}, using Lemma \ref{lem2152133} and \eqref{eq2171556} instead of Lemma \ref{lem2122217} and \eqref{eq2121719}, respectively. We omit the details.

We present the proof of $(i)$-$(ii)$ together. First, we deal with the a priori estimate \eqref{eq216207}. Let $v(t,x):=\eta_n(t)u(x):=\eta(t/n)u(x)$ where $\eta \in C_c^\infty(\bR)$ and $u\in C_c^\infty(\bR^d_+)$. Then $v$ satisfies the parabolic equation
    \begin{equation*}
      \sL_p v+\lambda c_0 v = \eta_n f + \eta_n' u
    \end{equation*}
    in $\bR\times \bR^d_+$.
    Note that for $g\in L_{p,\theta}$,
\begin{align*}
  \|\eta_n g\|_{\bL_{p,\theta}(\infty)}^p=nN_1\|g\|_{L_{p,\theta}}^p, \quad  \|\eta_n' g\|_{\bL_{p,\theta}(\infty)}^p=n^{1-p}N_2\|g\|_{L_{p,\theta}}^p,
\end{align*}
where
\begin{equation*}
  N_1:=\int_{-\infty}^\infty |\eta|^p \,dt, \quad N_2:=\int_{-\infty}^\infty |\eta'|^p \,dt.
\end{equation*}
Let $\lambda_0\geq0$ be taken from Theorem \ref{thmpara}. Then by \eqref{eq2111600} with the case $p=q$ and $\omega=1$, for $\lambda\geq\lambda_0$,
\begin{eqnarray*}
  &(1+\lambda)\|u\|_{L_{p,\theta}} + (1+\sqrt{\lambda})\|x_dD_xu\|_{L_{p,\theta}} + \|x_d^2D_x^2u\|_{L_{p,\theta}} \nonumber
  \\
  &\leq N \left( \|f\|_{L_{p,\theta}} + n^{-1}\|u\|_{L_{p,\theta}} \right).
\end{eqnarray*}
Thus, one can easily obtain \eqref{eq216207} by letting $n\to\infty$.

Lastly, we prove the existence. 
Thanks to the method of continuity, we may assume that the coefficients are constants. In addition, as in the proof of Theorem \ref{thmfinite}, we just need to prove the result for a single $\lambda\geq0$.

Note that the equation can be rewritten into a divergence form equation
   \begin{equation} \label{eq2152201}
    -x_d^2 D_{i}(\bar{a}_{ij} D_j u) + x_d b_i D_i u +cu +\lambda c_0 u= f,
   \end{equation}
   where
   \begin{equation*}
     \bar{a}_{ij}:=\left\{\begin{array}{ll}
       a_{ij}+a_{ji} &\text{ for } i\neq d \text{ and } j=d;
       \\
       0 &\text{ for } i= d \text{ and } j\neq d;
       \\
       a_{ij} &\text{ otherwise}.
     \end{array} \right.
   \end{equation*}
   By \cite[Theorem 2.14]{DR24div}, there is $\lambda_0\geq0$ such that for $\lambda\geq\lambda_0$, one can find a solution $u\in H_{p,\theta}^1$ to \eqref{eq2152201}. Moreover, for this solution, we have
   \begin{align} \label{eq2152239}
       (1+\lambda)\|u\|_{L_{p,\theta}} + (1+\sqrt{\lambda})\|x_dD_xu\|_{L_{p,\theta}} \leq N \|f\|_{L_{p,\theta}}.
   \end{align}
   
 Let $h_r(x):=h(x/r)$ for any function $h$ on $\bR^d_+$ and $v(x):=\zeta(x_d) u_r(x) := \zeta(x_d) u(x/r)$ where $\zeta\in C_c^\infty((2,3))$ is a standard nonnegative cutoff function. Then $v \in W_{p}^1(\bR^d)$ satisfies
    \begin{equation} \label{eq2172233}
    -x_d^2 a_{ij,r}D_{ij}v+ (\Lambda +\lambda) c_{0,r} v = g,
  \end{equation}
   where $\Lambda\geq0$ and $g$ is given by \eqref{eq2171558}.
Since $g\in L_{p}(\bR^d)$, one can apply \cite[Theorem 2.5]{D12} to obtain a solution $\tilde{v}$ to \eqref{eq2172233} in $W_{p}^{2}(\bR^d)$. Since \eqref{eq2172233} can be understood in divergence form, the uniqueness result in $W_{p}^1(\bR^d)$ leads to $v=\tilde{v}\in W_{p}^{2}(\bR^d)$, and
\begin{eqnarray} \label{eq2152242}
    (1+\lambda)\|v\|_{L_{p}(\bR^d)}+ (1+\sqrt{\lambda})\|D_x v\|_{L_{p}(\bR^d)} + \| D_x^2 v\|_{L_{p}(\bR^d)} \leq N \|g\|_{L_{p}(\bR^d)}.
\end{eqnarray}
Similar to \eqref{eq2122053}, we raise both sides of \eqref{eq2152242} to the power of $p$, multiply by $r^{-\theta-d}$, and integrate with respect to $r$. Then,
\begin{align*}
  &(1+\lambda)\|u\|_{L_{p,\theta}} + (1+\sqrt{\lambda})\|x_dD_xu\|_{L_{p,\theta}} + \|x_d^2D_x^2u\|_{L_{p,\theta}}
  \\
  &\leq N \left(\|f\|_{L_{p,\theta}} + \|u\|_{L_{p,\theta}} + \|x_dD_xu\|_{L_{p,\theta}}\right) \leq N\|f\|_{L_{p,\theta}}.
\end{align*}
Here, we used \eqref{eq2152239} for the last inequality.
Thus, $u\in H_{p,\theta}^2$, which proves the existence.
The theorem is proved.
\end{proof}

\section{Optimality of the weights when $\lambda=0$} \label{sectheta}

In this section, we prove that the ranges of $\theta$ in our main theorems are sharp when $\lambda=0$.

\subsection{Elliptic equations when $d=1$}

We show that $\theta\in \bR\setminus\{\alpha p,\beta p\}$ is optimal in Theorem \ref{thmell} $(iii)$.

Let us assume that the coefficients $a,b$, and $c$ are constants. We first show that \eqref{eq216207} fails to hold when $\theta$ is one root of \eqref{eq2152256}.
Let us consider the case $\theta=\alpha p$.
Let $u \in C_c^\infty(\bR_+)$. As in \eqref{eq2152249}, $v:=x^\alpha u$ satisfies
\begin{equation*}
     -x^2aD_x^2v + x (2a\alpha+b) D_xv = x^\alpha f.
\end{equation*}
If \eqref{eq216207} holds for $\theta=\alpha p$, then for any $\varepsilon>0$,
\begin{align} \label{eq8182344}
  \|v\|_{L_{p,0}} &=\|u\|_{L_{p,\alpha p}} \leq N \|f\|_{L_{p,\alpha p}} = N \|x^\alpha f\|_{L_{p,0}} \nonumber
  \\
  &\leq N \left(\|xD_xv\|_{L_{p,0}} + \|x^2D_x^2v\|_{L_{p,0}} \right) \nonumber
  \\
  &\leq \varepsilon\|v\|_{L_{p,0}} + N_{\varepsilon}\|x^2D_x^2v\|_{L_{p,0}},
\end{align}
Here, for the last inequality, we used \eqref{eq2152354}. By taking sufficiently small $\varepsilon$, we obtain
\begin{equation*}
    \|v\|_{L_{p,0}} \leq N\|x^2D_x^2v\|_{L_{p,0}},
\end{equation*}
which contradicts the optimality of Hardy's inequality (see e.g. Historical remark 5.5 in the preface of \cite{K85}). Hence, the claim is proved for $\theta=\alpha p$. For the case when $\theta=\beta p$, one just needs to repeat the above argument with $x^\beta u$ instead of $v=x^\alpha u$.

\subsection{Parabolic equations when $d=1$}

We show that the range $(\alpha p,\beta p)$ is optimal when $\lambda=0$ in Theorem \ref{thmpara} $(ii)$. By investigating the proof of Theorem \ref{thmell}, if the theorem holds true for some $\theta\in\bR$, then one can obtain the a priori estimate for the corresponding elliptic equations. From \textbf{1}, we deduce that we cannot solve the parabolic equations for $\theta=\alpha p$ or $\theta=\beta p$. Hence, we only consider the case $\theta\notin [\alpha p,\beta p]$.

Assume that the coefficients $a,b$, and $c$ are constants.
Similar to \eqref{eq2152249}, by considering $x^{-1/2-b/2a} u$ instead of $u$, we can also assume that $b=-a$. In this case, we still have the condition $c>0$ since \eqref{eq2152256} has two distinct real roots. Moreover, the range of $\theta$ is given by $\left(-p\sqrt{\frac{c}{a}},p\sqrt{\frac{c}{a}}\right)$.
Let $f\in C_c^\infty(\bR\times\bR_+)$ be a non-negative function such that $f=1$ in $(-1,0)\times (1,e)$.
If we denote $v(t,x):=u(t,e^x)$, then it satisfies
\begin{equation*}
  v_t=av_{xx}-cv+g
\end{equation*}
in $\bR^2$, where $g(t,x)=f(t,e^x)$. Thus, $v$ can be represented as
\begin{equation*}
  v(t,x)=\int_{-\infty}^t \int_{\bR} \frac{1}{2\sqrt{a\pi(t-s)}} e^{-\frac{(x-y)^2}{4a(t-s)}-c(t-s)} g(s,y)dyds.
\end{equation*}
For $t\geq1, x\geq1$, and $(s,y)\in(-1,0)\times (0,1)$, we have $1/(2t)\leq 1/(t-s)\leq 1/t$ and $(x-y)^2\leq x^2$, which lead to
\begin{equation*}
  -\frac{(x-y)^2}{4a(t-s)}-c(t-s) \geq -\frac{x^2}{4at}-ct-c, \quad \frac{1}{\sqrt{t-s}} \geq \frac{1}{\sqrt{2t}}.
\end{equation*}
Thus, one can find a positive constant $M$, independent of $t$ and $x$, such that
\begin{equation*}
  v(t,x)\geq Mt^{-1/2} e^{-\frac{x^2}{4at}-ct}, \quad t\geq1, x\geq1.
\end{equation*}
Let $\theta >p\sqrt{\frac{c}{a}}$. Since $2at\theta/p>1$ for $t>T_0:=\max\{1,p/(2a \theta)\}$,
\begin{align*}
 \int_{\{x\geq 1\}}t^{-1/2}e^{-\frac{p}{4at}(x-\frac{2at \theta}{p})^2} dx \geq \sqrt{a\pi}\int_{\{x\geq0\}} \frac{1}{\sqrt{a\pi t}}e^{-\frac{p}{4at}x^2} dx = \sqrt{\frac{a\pi}{p}}.
\end{align*}
Thus, for $T>T_0$,
\begin{align*}
  &\int_{\Omega_T} |u|^p x^{\theta-1}dxdt = \int_{(-\infty,T)\times\bR} |v|^p e^{\theta x} dxdt
  \\
  &\geq M^p \int_{(T_0,T)\times\{x\geq1\}} t^{-p/2}e^{-\frac{p}{4at}(x-\frac{2at \theta}{p})^2 + (\frac{a}{p}\theta^2-cp)t} dxdt
\\
  &= M^p \sqrt{\frac{a\pi}{p}} \int_{(T_0,T)} t^{-(p-1)/2} e^{(\frac{a}{p}\theta^2-cp)t} dt.
\end{align*}
This implies that if $\theta >p\sqrt{\frac{c}{a}}$, then the following estimate
\begin{equation*}
  \|u\|_{\bL_{p,\theta}(T)} \leq N\|f\|_{\bL_{p,\theta}(T)}
\end{equation*}
does not hold since the left-hand side diverges as $T\to \infty$. The case when $\theta <-p\sqrt{\frac{c}{a}}$ is similar.

\subsection{Parabolic and elliptic equations when $d\geq2$}

In this step, we prove that when $d\geq2$ and $\lambda=0$, the range $(\alpha p,\beta p)$ is sharp in Theorems \ref{thmpara} and \ref{thmell}.
As in \textbf{2}, by investigating the proof of Theorem \ref{thmell}, it suffices to show that the range is sharp for the elliptic equations.

We first show that the uniqueness fails to hold when $\theta>\beta p$.
We consider the operator $\sL_e u:=-x_d^2\Delta u + b_dx_d D_du +cu$, where $b_d$ and $c$ are constants. By considering $x^\gamma u$ instead of $u$, we further assume that $b_d=-1$. Then, as in \textbf{1}, we need the condition $c>0$ to guarantee that \eqref{eq2152256} has two distinct real roots, say $\alpha=-\sqrt{c}$ and $\beta=\sqrt{c}$.

Let $I_\nu(x_d)$ be the modified Bessel function of the first kind of order $\nu\in\bR$, which is defined as
\begin{equation*}
  I_\nu (x_d)=e^{-\nu\pi i/2}J_{\nu}(ix_d),
\end{equation*}
where $J_\nu$ is the Bessel function of the first kind of order $\nu$:
\begin{equation*}
  J_\nu(x_d):=\sum_{m=0}^\infty \frac{(-1)^m}{m!\Gamma(1+\nu+m)}\left(\frac{x_d}{2}\right)^{2m+\nu}
\end{equation*}
and
\begin{equation*}
  J_{-\nu}(x_d):=\sum_{m=0}^\infty \frac{(-1)^m}{m!\Gamma(1-\nu+m)}\left(\frac{x_d}{2}\right)^{2m-\nu}
\end{equation*}
for $\nu\geq0$.
We consider the modified Bessel function of the second kind of order $\nu\in\bR$, which is defined as
\begin{equation*}
  K_\nu(x_d)=\frac{\pi}{2}\frac{I_{-\nu}(x_d)-I_\nu(x_d)}{\sin (\nu\pi)}
\end{equation*}
if $\nu$ is not an integer, and
\begin{equation*}
  K_n(x_d)=\lim_{\nu\to n} K_\nu(x_d)
\end{equation*}
if $n$ is an integer. It is well known that (see e.g. 9.6.1, 9.6.8, 9.6.9, and 9.7.2 of \cite{AS48}), $K_\nu$ is a solution to
\begin{equation*}
  x_d^2 D_d^2 K_{\nu}(x_d) + x_d D_d K_{\nu}(x_d) -(x_d^2+\nu^2)K_{\nu}(x_d)=0,
\end{equation*}
and we have the following estimates
\begin{equation} \label{eq10311527}
  |K_{\nu}(x_d)| \leq Nx_d^{-|\nu|}, \quad x_d\leq1, \, \nu\neq0,
\end{equation}
\begin{equation} \label{eq11041209}
  |K_{0}(x_d)| \leq -N\log x_d, \quad x_d\leq1,
\end{equation}
and
\begin{equation} \label{eq10311528}
  |K_{\nu}(x_d)| \leq Nx_d^{-1/2}e^{-x_d}, \quad x_d\geq1.
\end{equation}

Let $p\geq 2$. For a non-zero function $\eta\in C_c^\infty(\bR^{d-1})$ such that $\eta(\xi')=0$ near $\xi'=0$, we denote
\begin{equation*}
  v(\xi',x_d):=K_{\sqrt{c}}(|\xi'|x_d) \eta(\xi'),
\end{equation*}
which satisfies
\begin{equation} \label{eq1131512}
  x_d^2D_d^2v+x_dD_dv-(|\xi'|^2x_d^2+c)v=0.
\end{equation}
Let us consider
\begin{equation*}
  u(x)=u(x',x_d):= \frac{1}{(2\pi)^{(d-1)/2}}\int_{\bR^{d-1}} e^{i\xi'\cdot x'} v(\xi',x_d) d\xi',
\end{equation*}
which is well defined due to \eqref{eq10311527} and \eqref{eq10311528}.
Then by \eqref{eq1131512}, $u$ is a non-zero solution to $\sL_e u=0$. Thus, it remains to show that $u\in H_{p,\theta}^2$. By the Hausdorff-Young inequality in the $x'$-variable, the Minkowski inequality, \eqref{eq10311527}, and \eqref{eq10311528},
\begin{align*}
  &\int_{\bR^d_+} |u(x)|^p x_d^{\theta-1} dx
  \\
  &\leq N \int_0^\infty \left( \int_{\bR^{d-1}} |v(\xi',x_d)|^{p'} d\xi' \right)^{p/p'} x_d^{\theta-1} dx_d
  \\
  &\leq N \left(\int_{\bR^{d-1}} \left(\int_{0}^{\infty} |v(\xi',x_d)|^{p} x_d^{\theta-1} d{x_d}\right)^{p'/p} d\xi' \right)^{p/p'}
  \\
  &\leq N \left(\int_{\bR^{d-1}} |\eta(\xi')|^{p'}|\xi'|^{-\theta p'/p} \left(\int_{0}^{\infty} |K_{\sqrt{c}}(x_d)|^p x_d^{\theta-1} d{x_d}\right)^{p'/p} d\xi' \right)^{p/p'} <\infty,
\end{align*}
where $p'=p/(p-1)$.
One can also notice that $v(x):=\zeta(x_d) u(x/r)$ satisfies \eqref{eq8151424} with $f_r=0$, where $\zeta\in C_c^\infty((2,3))$ is a standard nonnegative cutoff function. Then by the uniqueness result of \cite[Theorem 2.5]{D12}, we have $v\in W_p^2(\bR^d)$.
Thus, we can repeat the proof of Lemma \ref{lem8150107} to obtain that
  \begin{align*}
  &\int_{\bR_+} |xD_xu|^p x^{\theta-1} dx + \int_{\bR_+} |x^2D_{x}^2u|^p x^{\theta-1} dx \leq N\int_{\bR_+} |u|^{p} x^{\theta-1} dx.
  \end{align*}
 Since this inequality holds for any $\theta\in\bR$, we obtain $u\in H_{p,\theta}^2$. Thus, the claim is proved when $p\geq2$.

Next, we consider the case $p<2$. 
Let $\zeta_n(x_d)$ be a cut-off function such that $\zeta_n=1$ on $(1/n,n)$, $\zeta_n=0$ on $(0,1/(2n))\cup(2n,\infty)$, $x_d|\zeta_n'| \leq N$, and $x_d^2|\zeta_n''| \leq N$. Then by H\"older's inequality, for a sufficiently large $m\in\bN$,
\begin{align} \label{eq11040029}
  &\int_{\bR^d_+} \left(|u\zeta_n|^p + |x_dD_x(u\zeta_n)|^p + |x_d^2D_x^2(u\zeta_n)|^p \right) x_d^{\theta-1} dx \nonumber
  \\
  &\leq N \left(\int_{\bR^d_+} \left(|u|^2 + |x_dD_xu|^2 + |x_d^2D_x^2u|^2 \right) \left(1+|x'|^2\right)^m x_d^{2\theta/p -1} dx\right)^{p/2} \nonumber
  \\
  &\qquad\quad\times\left(\int_{A_n} \left(1+|x'|^2\right)^{mp/(p-2)} x_d^{-1} dx\right)^{(2-p)/2} \nonumber
    \\
  &\leq N(n) \left(\int_{\bR^d_+} \left(|u|^2 + |x_dD_xu|^2 + |x_d^2D_x^2u|^2 \right) \left(1+|x'|^2\right)^m x_d^{2\theta/p -1} dx\right)^{p/2},
\end{align}
where $A_n=\bR^{d-1}\times (1/(2n),2n)$. By the Plancherel theorem,
\begin{align} \label{eq11011849}
  &\int_{\bR^d_+} \left(|u|^2 + |x_dD_xu|^2 + |x_d^2D_x^2u|^2 \right) \left(1+|x'|^2\right)^m x_d^{2\theta/p -1} dx \nonumber
  \\
  &\leq N(d,m) \sum_{0\leq i\leq j \leq 2} \sum_{k=0}^m \int_{\bR^d_+} |\xi'|^{2i} |x_d^j D_d^{j-i} D_{\xi'}^kv|^2 x_d^{2\theta/p -1} d\xi' dx_d.
\end{align}
Since $\eta(\xi')=0$ near $\xi'=0$, when $\eta(\xi')\neq0$,
\begin{align*}
  D_{\xi'}^l |\xi'| \leq N(\eta,l)
\end{align*}
for any $l\in \bN$. Thus, using \eqref{eq10311527}, \eqref{eq11041209}, \eqref{eq10311528}, and the relations $K_{-\nu}=K_\nu$ and
\begin{equation*}
  -2K_{\nu}'=K_{\nu-1}+K_{\nu+1}
\end{equation*}
(see e.g. 9.6.6 and 9.6.26 of \cite{AS48}),
for $0\le i\le j\le 2$ and $\xi'\in \bR^{d-1}$ such that $\eta(\xi')\neq0$,
\begin{align} \label{eq11040005}
  |x_d^{j} D_d^{j-i}D_{\xi'}^k\left(K_{\sqrt{c}}(|\xi'|x_d)\right)| &\leq N \sum_{l=0}^{k} x_d^{l+j} |K_{\sqrt{c}}^{(l+j-i)}(|\xi'|x_d)| \nonumber
  \\
  &\leq N \sum_{l=-k-2}^{k+2}  \sum_{|l_0|\leq l}(x_d^{|l|}+x_d^{|l|+2}) |K_{\sqrt{c}+l_0}(|\xi'|x_d)| \nonumber
  \\
  &\leq N\sum_{l=-k-2}^{k+2}  (x_d^{|l|}+x_d^{|l|+2}) |\widetilde{K}_{\sqrt{c}+l}(|\xi'|x_d)|,
\end{align}
where
\begin{equation*}
 \widetilde{K}_{\nu}(x_d):=
    \begin{cases}
       x_d^{-|\nu|}, \quad &x_d\leq 1,\nu\neq0,
        \\
        -\log x_d, \quad &x_d\leq1, \nu=0,
        \\
        e^{-x_d/2}, \quad &x_d>1.
    \end{cases}
\end{equation*}
Since $2\theta/p>2\beta=2\sqrt{c}$, by \eqref{eq11040005},
\begin{equation} \label{eq11011850}
  \sum_{0\leq i \leq j\leq2} \sum_{k=0}^m \int_{\bR^d_+} |\xi'|^{2i} |x_d^j D_d^{j-i} D_{\xi'}^kv|^2 x_d^{2\theta/p - 1} d\xi' dx_d \leq N(\eta).
\end{equation}
Thus, by \eqref{eq11040029}, \eqref{eq11011849}, and \eqref{eq11011850}, we have $u\zeta_n \in H_{p,\theta}^2$.

Suppose that the a priori estimate \eqref{eq216207} holds true when $\theta>\beta p$.
Since $u\zeta_n\in H_{p,\theta}^2$, $\sL_e u=0$, and
\begin{align*}
  \sL_e(u\zeta_n) = -2x_d^2\zeta_n'D_du - x_d^2\zeta_n''u -x_d\zeta_n'u,
\end{align*}
by \eqref{eq216207} and H\"older's inequality, for a sufficiently large $m\in\bN$,
\begin{align} \label{eq11011848}
  &\int_{\bR^d_+} \left(|u\zeta_n|^p + |x_dD_x(u\zeta_n)|^p + |x_d^2D_x^2(u\zeta_n)|^p \right) x_d^{\theta-1} dx \leq N\int_{\bR^d_+} |\sL_e(u\zeta_n)|^p x_d^{\theta-1} dx \nonumber
  \\
  &\leq N\int_{B_n} \left(|u|^p + |x_dD_xu|^p + |x_d^2D_x^2u|^p \right) x_d^{\theta-1} dx \nonumber
  \\
  &\leq N \left(\int_{\bR^d_+} \left(|u|^2 + |x_dD_xu|^2 + |x_d^2D_x^2u|^p \right) \left(1+|x'|^2\right)^m x_d^{2\theta/p -1} dx\right)^{p/2} \nonumber
  \\
  &\qquad\quad\times\left(\int_{B_n} \left(1+|x'|^2\right)^{mp/(p-2)} x_d^{-1} dx\right)^{(2-p)/2} \nonumber
    \\
  &\leq N \left(\int_{\bR^d_+} \left(|u|^2 + |x_dD_xu|^2 + |x_d^2D_x^2u|^2 \right) \left(1+|x'|^2\right)^m x_d^{2\theta/p -1} dx\right)^{p/2} \nonumber
  \\
  &\qquad\quad\times (\log 2)^{(2-p)/2}\left(\int_{\bR^{d-1}} \left(1+|x'|^2\right)^{mp/(p-2)} dx'\right)^{(2-p)/2} \nonumber
  \\
  &\leq N \left(\int_{\bR^d_+} \left(|u|^2 + |x_dD_xu|^2 + |x_d^2D_x^2u|^2 \right) \left(1+|x'|^2\right)^m x_d^{2\theta/p -1} dx\right)^{p/2},
\end{align}
where $B_n:=\bR^{d-1} \times \left((1/(2n),1/n)\cup (n,2n)\right)$, and $N$ is independent of $n$. 
Combining \eqref{eq11011849}, \eqref{eq11011850}, and \eqref{eq11011848},
\begin{align*}
  \int_{\bR^d_+} \left(|u\zeta_n|^p + |D_x(u\zeta_n)|^p + |x_d^2D_x^2(u\zeta_n)|^p \right) x_d^{\theta-1} dx \leq N(\eta).
\end{align*}
By letting $n\to \infty$, we obtain that $u \in H_{p,\theta}^2$, and it is a non-zero solution to $\sL_e u=0$, which is a contradiction.

Now we use a duality argument to prove that Theorem \ref{thmell} does not hold when $\theta<\alpha p$. As above, we still consider the operator $\sL_e u:= -x_d^2\Delta u -x_dD_du + cu$ and the case $\alpha=-\sqrt{c}$. One can show that
\begin{equation*}
    \sL_e^* u:= -x_d^2\Delta u -3x_dD_du + (c-1)u
\end{equation*}
is the dual operator of $\sL_e$, and the corresponding quadratic equation for $\sL_e^*$ is
\begin{equation*}
    z^2-2z-(c-1)=0,
\end{equation*}
whose two real roots are $\sqrt{c}+1$ and $-\sqrt{c}+1$. Our approach is to derive a contradiction by assuming that the theorem holds when $\theta<\alpha p=\sqrt{c}p$ and then showing that, in this case, we obtain the uniqueness result (or a priori estimate) for $\sL_e^*$ when $\theta'>(\sqrt{c}+1)p'$, where $p'=p/(p-1)$ and $\theta/p+\theta'/p'=1$.

Let $g \in C_c^\infty(\bR^d_+)$ and $\theta<\alpha p=\sqrt{c}p$. Then by the assumption, there is a solution $v\in H_{p,\theta}^2$ to $\sL_e v=g$. Then for $u\in C_c^\infty(\bR^d_+)$,
\begin{align*}
    \left|\int_{\bR^d_+} ug \,dx\right| &= \left|\int_{\bR^d_+} u\sL_e v \,dx\right| = \left|\int_{\bR^d_+} v\sL_e^* u \,dx\right|
    \\
    &\leq \|v\|_{L_{p,\theta}} \|\sL_e^* u\|_{L_{p',\theta'}} \leq N\|g\|_{L_{p,\theta}} \|\sL_e^* u\|_{L_{p',\theta'}}.
\end{align*}
Here, for the last inequality, we used \eqref{eq216207} for $v$. Since $g$ is arbitrary, we obtain
\begin{equation*}
    \|u\|_{L_{p',\theta'}} \leq N\|\sL_e^*u\|_{L_{p',\theta'}},
\end{equation*}
which yields the uniqueness for $\sL_e^*$. Thus, our goal is obtained.

Lastly, we consider the case $\theta=\alpha p$ or $\beta p$. Due to a duality argument as above, we only show that
\begin{align} \label{eq10231745}
   \int_{\bR_+^d} |u|^p x_d^{\theta-1} dx\leq N \int_{\bR_+^d} |\sL_e u|^p x_d^{\theta-1} dx, \quad u\in C_c^\infty(\bR^d_+)
\end{align}
fails to hold when $\theta=\alpha p$. Note that since $u\in C_c^\infty(\bR^d_+)$, $\sL_e u$ is well defined on $\bR^d_+$.
Let
\begin{equation*}
  u(x',x_d) := x_d^{-\alpha}w(x')\zeta(x_d),
\end{equation*}
where $w\in C_c^\infty(\bR^{d-1})$, and $\zeta\in C^\infty((0,\infty))$ such that $\eta=1$ on $(0,1)$, and $\zeta=0$ on $(2,\infty)$. Then due to \eqref{eq2152256},
\begin{align*}
  \sL_e u&:=-x_d^2\Delta u -x_d D_du +cu 
  \\
  &= -x_d^{-\alpha+2} \zeta \Delta_{x'}w - x_d^{-\alpha+2} \zeta'' w +\left(2\alpha -1 \right)x_d^{-\alpha+1} \zeta' w
  \\
  &\quad +\left( -\alpha(\alpha+1)+\alpha + c\right)x_d^{-\alpha} \zeta w
  \\
  &= -x_d^{-\alpha+2} \zeta \Delta_{x'}w - x_d^{-\alpha+2} \zeta'' w +\left(2\alpha -1 \right)x_d^{-\alpha+1} \zeta' w =:f(x).
\end{align*}
Here, since $|f| \approx x_d^{-\alpha+2}$ near $x_d=0$, and $f=0$ near $x_d=\infty$,  $f\in L_{p,\theta}$.

By using this example, we show that we cannot obtain the a priori estimate \eqref{eq10231745} when $\theta= \alpha p$. 
Let $\eta_n(x_d)$ is a cut-off function such that $\eta_n=1$ on $(1/n,\infty)$, $\eta_n=0$ on $(0,1/(2n))$, $x_d|\eta_n'| \leq N$, and $x_d^2|\eta_n''| \leq N$.
For $\delta\in(0,1)$, we further define $\tau_n(x_d):=\eta_n(x_d^{\delta})$, which leads to
\begin{equation*}
  x_d|\tau_n'| \leq N\delta, \quad x_d^2|\tau_n''| \leq N\delta.
\end{equation*}
Then both $u\tau_n$ and
\begin{align*}
  f_n := \sL_e(u\tau_n) = \tau_n f -2x_d^2\tau_n'D_du - x_d^2\tau_n''u -x_d\tau_n'u
\end{align*}
 are in the space $L_{p,\theta}$. Note that
 \begin{align*}
   \int_{\bR_+^d} |x_d^2\tau_n'D_du|^p x_d^{\theta-1} dx \leq N\delta^p \int_{(2n)^{-1/\delta}}^{1} x_d^{-1} dx_d \leq N \delta^{p-1} \log(2n).
 \end{align*}
Thus, by similar computations,
\begin{align} \label{eq10231720}
  \|f_n\|_{L_{p,\theta}}^p &\leq N \left( \|\tau_n f\|_{L_{p,\theta}}^p + \|x_d^2\tau_n'D_xu\|_{L_{p,\theta}}^p + \|x_d^2\tau_n''u\|_{L_{p,\theta}}^p + \|x_d\tau_n'u\|_{L_{p,\theta}}^p \right) \nonumber
  \\
  &\leq N\left(\|f\|_{L_{p,\theta}}^p + \delta^{p-1} \log(2n)\right).
\end{align}
On the other hand,
\begin{align} \label{eq10231721}
  \int_{\bR_+^d} |\tau_n u|^p x_d^{\theta-1} dx \geq N\int_{n^{-1/\delta}}^1 x_d^{-1} dx_d = N\delta^{-1}\log n
\end{align}
Combining \eqref{eq10231745}, \eqref{eq10231720}, and \eqref{eq10231721}, 
\begin{align*}
  \delta^{-1}\log n &\leq N\left(\|f\|_{L_{p,\theta}}^p + \delta^{p-1} \log(2n)\right),
\end{align*}
which is equivalent to
\begin{align*}
  \delta^{-1} &\leq N\left(\frac{1}{\log n}\|f\|_{L_{p,\theta}}^p + \delta^{p-1} \frac{\log(2n)}{\log n}\right).
\end{align*}
Since $\|f\|_{L_{p,\theta}}<\infty$, letting $n\to \infty$ and $\delta\to0$ in order, we arrive at a contradiction.

\begin{remark}
Let us discuss the optimality results for the corresponding divergence form equations introduced in \cite{DR24div}. If we apply $(2.11)$ in \cite[Theorem 2.14]{DR24div}, which is the divergence form version of \eqref{eq216207}, we still obtain \eqref{eq8182344}. Hence, the desired result follows for the one-dimensional elliptic equations. Furthermore, the counterexamples in \textbf{2} and \textbf{3} also imply the optimality of \cite[Theorems 2.6 and 2.14]{DR24div}.
\end{remark}

\section{Application: degenerate viscous Hamilton-Jacobi equations} \label{sec_app}

In this section, we consider a degenerate viscous Hamilton–Jacobi equation as an application of the main result. More precisely, we study
\begin{equation} \label{eq3231213}
    u_t +\lambda u -x_d^2\Delta u + x_d b_iD_iu + cu = H(t,x,D_xu), \text{ in } \Omega_T,
\end{equation}
where $H$ is a given Hamiltonian, and $b_i,c$, and $\lambda\geq0$ are constant. Here, we assume that there exist a constant $C>0$ and a measurable function $h$ defined on $\Omega_T$ such that for any $(t,x)\in \Omega_T$ and $P\in\bR^d$,
\begin{equation} \label{eq3241451}
    |H(t,x,P)| \leq C\left(\min\{x_d,1\}|P| + h(t,x)\right).
\end{equation}

Now we present a regularity result for solutions to \eqref{eq3231213}. In the following theorem, we consider $\frH_{p,\theta}^2(T):=\frH_{p,p,\theta,1}^2(T)$ and $\bL_{p,\theta}(T):=\bL_{p,p,\theta,1}(T)$.

\begin{theorem}
    Let $T\in(-\infty,\infty]$, $p\in(1,\infty)$, and $\theta\in\bR$. Then there exists a number $\lambda_0=\lambda_0(d,p,\theta,b_i,c,C)\geq0$ such that for any $\lambda\geq\lambda_0$ and $h\in \bL_{p,\theta}(T)$, there is a unique solution $u\in \frH_{p,\theta}^2(T)$ to \eqref{eq3231213}. Moreover, for this solution,
    \begin{eqnarray} \label{eq3241444}
        &\|u_t\|_{\bL_{p,\theta}(T)} + (1+\lambda)\|u\|_{\bL_{p,\theta}(T)} + (1+\sqrt{\lambda})\|x_dD_xu\|_{\bL_{p,\theta}(T)} + \|x_d^2D_x^2u\|_{\bL_{p,\theta}(T)}& \nonumber
        \\
        &\leq N \|h\|_{\bL_{p,\theta}(T)},&
    \end{eqnarray}
    where $N=N(d,p,\theta,b_i,c,C)$.
\end{theorem}

\begin{proof}
    Due to the method of continuity, we only need to prove \eqref{eq3241444}. Suppose that $u\in \frH_{p,\theta}^2(T)$ is a solution to \eqref{eq3231213}. By Theorem \ref{thmpara} $(i)$, there is $\lambda_1\geq0$ such that for any $\lambda\geq\lambda_1$,
    \begin{eqnarray*}
        &\|u_t\|_{\bL_{p,\theta}(T)} + (1+\lambda)\|u\|_{\bL_{p,\theta}(T)} + (1+\sqrt{\lambda})\|x_dD_xu\|_{\bL_{p,\theta}(T)} + \|x_d^2D_x^2u\|_{\bL_{p,\theta}(T)}&
        \\
        &\leq N \|H\|_{\bL_{p,\theta}(T)}.&
    \end{eqnarray*}
    By \eqref{eq3241451}, \eqref{eq3241450} below, and Young's inequality, for any $\varepsilon>0$,
    \begin{align*}
        \|H\|_{\bL_{p,\theta}(T)} &\leq N \left( \|\min\{x_d,1\}|D_xu|\|_{\bL_{p,\theta}(T)} + \|h\|_{\bL_{p,\theta}(T)} \right)
        \\
        &\leq N\left(\|u\|_{\bL_{p,\theta}(T)}^{1/2} \|x_d^2D_x^2u\|_{\bL_{p,\theta}(T)}^{1/2} + \|u\|_{\bL_{p,\theta}(T)} + \|h\|_{\bL_{p,\theta}(T)}\right)
        \\
        &\leq N\varepsilon \|x_d^2D_x^2u\|_{\bL_{p,\theta}(T)} + N(\varepsilon)\|u\|_{\bL_{p,\theta}(T)} + N \|h\|_{\bL_{p,\theta}(T)}.
    \end{align*}
    By taking $\varepsilon>0$ so that $N\varepsilon<1/2$, we have
    \begin{eqnarray*}
        &\|u_t\|_{\bL_{p,\theta}(T)} + (1+\lambda)\|u\|_{\bL_{p,\theta}(T)} + (1+\sqrt{\lambda})\|x_dD_xu\|_{\bL_{p,\theta}(T)} + \|x_d^2D_x^2u\|_{\bL_{p,\theta}(T)}&
        \\
        &\leq N\left(\|u\|_{\bL_{p,\theta}(T)} + N \|h\|_{\bL_{p,\theta}(T)}\right).&
    \end{eqnarray*}
    Thus, there is $\lambda_2\geq0$ so that if $\lambda\geq\lambda_2$, we can absorb $\|u\|_{\bL_{p,\theta}(T)}$ on the right-hand side to the left-hand side. Hence, the claim is obtained for $\lambda_0:=\max\{\lambda_1,\lambda_2\}$. The proof is completed.
\end{proof}

\begin{remark}
    $(i)$ A related result for \eqref{eq3231213}, with the diffusion term $x_d^\alpha\Delta$ ($\alpha\in(0,2)$) is given in \cite[Section 6]{DPT24}. Compared to this result, there is no restriction on the weight parameter $\theta$.

    $(ii)$ Obviously, one can obtain the same regularity result when $H$ satisfies
    \begin{equation*}
        |H(t,x,P)| \leq C\left(\min\{x_d^\delta,1\}|P| + h(t,x)\right)
    \end{equation*}
    where $\delta\geq1$. However, the case $\delta\in(0,1)$ remains open, including the less degenerate case $x_d^\alpha \Delta$ with $\alpha\in(0,2)$.
    
    $(iii)$ In the literature, Hamiltonians more general than the specific form \eqref{eq3241451} are typically of greater interest. For general degenerate viscous Hamilton-Jacobi equations, Lipschitz a priori estimates for solutions are available (see \cite{CIL92, AT15, LMT17} and the references therein). However, for general $H$, stronger regularity results are not well understood, and optimal regularity of solutions near the boundary $\{x_d = 0\}$ has not been investigated.
\end{remark}

The following lemma can be obtained by repeating the proof of \cite[Lemma 6.3]{DPT24}. We provide the proof for the reader's convenience.

\begin{lemma}
    Let $p\in(1,\infty)$ and $\theta \in \bR$. Then for any $u\in C_c^\infty((-\infty,T]\times \bR^d_+)$, $v:=\min\{x_d,1\}|D_xu|$ satisfies the following interpolation inequality
    \begin{align} \label{eq3241450}
        \|v\|_{\bL_{p,\theta}(T)} \leq N\left(\|u\|_{\bL_{p,\theta}(T)}^{1/2} \|x_d^2D_x^2u\|_{\bL_{p,\theta}(T)}^{1/2} + \|u\|_{\bL_{p,\theta}(T)}\right),
    \end{align}
    where $N=N(d,p,\theta)$.
\end{lemma}

\begin{proof}
Let $
\Omega_m=\{(t,x)\in \Omega_T\,:\, 2^{-m-1} < x_d \leq 2^{-m}\}$ with $m\in\bZ$.
By the Gagliado-Nirenberg interpolation inequality, for each $m\in \bZ$, we get
\begin{equation*}
    \|D_xu\|_{L_{p}(\Omega_m)} \leq N \left (\|u\|_{L_p(\Omega_m)}^{1/2} \|D_x^2u\|_{L_p(\Omega_m)}^{1/2} + 2^{m} \|u\|_{L_p(\Omega_m)} \right).
\end{equation*}
For $m\geq0$, $v= x_d |D_xu|$ and
\begin{align*}
&\|v1_{\Omega_m}\|_{\bL_{p,\theta}(T)}^p = \|x_d|Du|1_{\Omega_m}\|_{\bL_{p,\theta}(T)}^p
= \int_{\Omega_m} |Du|^{p} x_d^{p+\theta-1} dxdt
\\
&\leq N 2^{-m(p+\theta-1)} \int_{\Omega_m} |D_xu|^{p} dxdt
\\
&\leq   N 2^{-m(p+\theta-1)} \left(\int_{\Omega_m} |u|^{p} dxdt \right)^{1/2} \left(\int_{\Omega_m} |D_x^2u|^{p} dxdt \right)^{1/2} 
\\
&\qquad+ N  2^{-m(\theta-1)} \int_{\Omega_m} |u|^{p} dxdt
\\
&\leq  N \left(\|u1_{\Omega_m}\|_{\bL_{p,\theta}(T)}^{p/2}\|x_d^2D_x^2 u1_{\Omega_m}\|_{\bL_{p,\theta}(T)}^{p/2} + \|u1_{\Omega_m}\|_{\bL_{p,\theta}(T)}^{p} \right).
\end{align*}
Similarly, for $m<0$,
\begin{align*}
\|v1_{\Omega_m}\|_{\bL_{p,\theta}(T)}^p \leq  N 2^{mp} \left(\|u1_{\Omega_m}\|_{\bL_{p,\theta}(T)}^{p/2}\|x_d^2D_x^2 u1_{\Omega_m}\|_{\bL_{p,\theta}(T)}^{p/2} + \|u1_{\Omega_m}\|_{\bL_{p,\theta}(T)}^{p} \right).
\end{align*}
Thus, we obtain
\begin{align*}
    \|v\|_{\bL_{p,\theta}(T)}^p = \sum_{m\in\bZ} \|v1_{\Omega_m}\|_{\bL_{p,\theta}(T)}^p \leq N\left(\|u\|_{\bL_{p,\theta}(T)}^{p/2} \|x_d^2D_x^2u\|_{\bL_{p,\theta}(T)}^{p/2} + \|u\|_{\bL_{p,\theta}(T)}^p\right).
\end{align*}
The lemma is proved.
\end{proof}


\section*{Acknowledgement} \pdfbookmark{Acknowledgement}{Acknowledgement}

 The authors would like to thank the anonymous referee for his/her very useful comments.



\end{document}